\documentclass{svjour3}

\usepackage{amsmath}
\usepackage{amsfonts}
\usepackage{amssymb}
\usepackage{diagbox}
\usepackage{tkz-graph}

\usepackage{subcaption}
\usepackage{enumitem}
\usepackage{bbm}

\usepackage{cleveref}
\usepackage{mathtools}
\usepackage{tikz}
\usetikzlibrary{positioning,chains,fit,shapes,calc}
\begin{document}

\title{Exploiting Partial Correlations in Distributionally Robust Optimization}

\author{Divya Padmanabhan \and  Karthik Natarajan \and Karthyek R. A. Murthy}
\institute{Email: divya\_padmanabhan$@$sutd.edu.sg, karthik\_natarajan$@$sutd.edu.sg, karthyek\_murthy$@$sutd.edu.sg. Singapore University of Technology and Design}

\date{Submitted: October 2018}

\maketitle

\begin{abstract}
In this paper, we identify partial correlation information structures that allow for simpler reformulations in evaluating the maximum expected value of mixed integer linear programs with random objective coefficients. To this end, assuming only the knowledge of the mean and the covariance matrix entries restricted to block-diagonal patterns, we develop a reduced semidefinite programming formulation, the complexity of solving which is related to characterizing a suitable projection of the convex hull of the set $\{(\bold{x}, \bold{x}\bold{x}'): \bold{x} \in \mathcal{X}\}$ where $\mathcal{X}$ is the feasible region. In some cases, this lends itself to efficient representations that result in polynomial-time solvable instances, most notably for the distributionally robust appointment scheduling problem with random job durations as well as for computing tight bounds in Project Evaluation and Review Technique (PERT)
networks and linear assignment problems. To the best of our knowledge, this is the first example of a distributionally robust optimization formulation for appointment scheduling that permits a tight polynomial-time solvable semidefinite programming reformulation which explicitly captures partially known correlation information between uncertain processing times of the jobs to be scheduled.
\end{abstract}

\maketitle \section{Introduction} \label{sec:intro} We consider
decision problems where the objective involves maximizing the expected
value of $Z(\bold{\tilde{c}}),$ where
$\bold{\tilde{c}} = (\tilde{c}_1, \tilde{c}_2, \ldots, \tilde{c}_n)$ is a $n$-dimensional real
valued random vector, such that,
\begin{align}
  Z(\tilde{\bold{c}}) = \max \left\{ \tilde{\bold{c}}' \bold{x}: \bold{x}
  \in \mathcal{X} \right\},
\label{eq:inner-obj}
\end{align}
and the set $\mathcal{X}$ is the bounded feasible region to a mixed integer linear program (MILP):
\begin{align*}
  \mathcal{X} = \left\{ \bold{x} \in \mathbb{R}^n: \bold{A}\bold{x} =
  \bold{b}, \ \bold{x} \geq 0,\ x_j \in \mathcal{Z} \text{ for }
  j \in \mathcal{I} \subseteq [n] \right\}.
\end{align*}
The set $\mathcal{X}$ includes the feasible region to linear optimization problems as a special case. The distribution $\theta$ of $\bold{\tilde{c}}$ is not always known explicitly, while many a time, only a set $\mathcal{P}$ of distributions is known such that $\theta \in \mathcal{P}$. In this scenario, we are interested in computing the quantity $\sup\{\mathbb{E}_{\theta}[ Z(\tilde{\bold{c}})] : \mathbb{\theta} \in \mathcal{P} \}$, referred to as the distributionally robust bound.
In this paper, we focus on the case where only the first moment of
$\bold{\tilde{c}}$ along with some of the second moments are specified. Applications where such bounds have been previously studied include appointment scheduling, portfolio management and the newsvendor problem among others. For more details, the interested reader may refer to \cite{ben2009robust,Delage2010,Bertsimas2010,Natarajan2011,Kong2013,Wiese2014,natarajan2017reduced}.

A precise description of the problem is provided next. Suppose that $\mathcal{N}_1, \ldots, \mathcal{N}_R$ form a partition
of the set $\mathcal{N}=\{1, \ldots, n\},$ so that
$\mathcal{N} = \bigcup_{r} \mathcal{N}_r$ and
$\mathcal{N}_i \, \cap \, \mathcal{N}_j =\varnothing$ for $i \neq j.$
We use $n_r = |\mathcal{N}_r|$ to denote the size of the subset
$\mathcal{N}_r.$ For any vector $\bold{a} \in \mathbb{R}^n,$ let
$\bold{a}^r \in \mathbb{R}^{n_r}$ be the subvector formed using
elements in $\mathcal{N}_r$ as indices. Let
$\mathcal{P}(\mathbb{R}^n)$ be the set of probability distributions on
$\mathbb{R}^n.$ Suppose that the only information we know about the
probability distribution of $\tilde{\bold{ c}}$ is the first moment specified by $\mathbb{E}[\tilde{\bold{c}}] = \boldsymbol{\mu}$ and the second moment matrices $\mathbb{E}[\tilde{\bold{c}}^r(\tilde{\bold{c}}^r)'] = \bold{\Pi}^r$
for $r \in [R] = \{1,\ldots,R\}.$ In this situation, we are interested in:
\begin{align}
  Z^* = \sup \left\lbrace
  \mathbb{E}_\theta\left[Z(\bold{\tilde{c}})\right]:    \begin{array}{ll}
          \mathbb{E}_{\theta}[\bold{\tilde{c}}] =\boldsymbol{
          \mu}, \ \mathbb{E}_{\theta} [\tilde{\bold{c}}^r(\tilde{\bold{c}}^r)']  =
          \bold{\Pi}^r \text{ for } r \in [R\,], \ \theta \in
          \mathcal{P}(\mathbb{R}^n)
\end{array}
 \right\rbrace,
\label{Moment-Problem}
\end{align}
which quantifies the maximum possible expected value of
$Z(\bold{\tilde{c}})$ over all probability distributions $\theta$
whose first and second moments are consistent with the moment
information specified for the random vector $\bold{\tilde{c}}.$ We assume that all $r \in [R\,], \bold{\Pi}^r  \succ \boldsymbol{\mu}^r {\boldsymbol{\mu}^r}'$, which is sufficient to guarantee that strong duality holds and in the resulting dual formulations, the optimum is attained. Since $R$ denotes the number of non-overlapping subsets, the
partition for $R=n$ corresponds to the case where only the mean and
diagonal (variance) entries of the covariance matrix are specified. On
the other hand, $R = 1$ corresponds to the case
where the mean and the entire covariance matrix is specified. Hence, $\bold{\Pi}^r$'s denote known sub-matrices of $\bold{\Pi}$ which denotes the matrix of all second moments of $\tilde{\bold{{c}}}$. Thus, we relax the assumption that the complete matrix $\bold{\Pi}$ is known for a fixed $R > 1$, but only that some entries are known.
The model studied in this paper is closest to the model analyzed in \cite{Doan2012}. Therein, the authors studied the distributionally robust bound $\sup\{ \mathbb{E}_{\theta} [\max {\tilde{\bold{c}}}' \bold{x}: \bold{x} \in \mathcal{X} \subset \{0,1\}^n] : \theta \in {\cal P} \} $ where multivariate marginal discrete distributions of non-overlapping subsets of random variables are specified. While the bound is NP-hard to compute, \cite{Doan2012} identified two instances of the problem for subset selection and Project Evaluation and Review (PERT) networks, where the tight bound is computable in polynomial-time.
We build on the model in \cite{Doan2012} by allowing for decision variables $\bold{x} \in \mathcal{X} \subseteq \mathbb{R}^n$ and considering moment-based ambiguity sets.

\textbf{Notations.}  Let $\mathbb{R}^{m \times n}$ be the set of
$m\times n$ matrices with real entries, $\mathcal{S}^{k}$ be the
set of $k \times k$ symmetric matrices and $\mathcal{S}_+^{k}$ be the
set of $k \times k$ symmetric positive semidefinite (psd) matrices. We write
$\bold{A} \succeq 0$ to denote that $\bold{A}$ is a psd matrix. For
any positive integer $k,$ we use $[k]$ to denote the set
$\{1,2,\ldots,k\}.$ For any subset $I$ of $[k]$ and matrix
$\bold{A} \in \mathbb{R}^{k \times k},$ we use $\bold{A}[I]$ to denote
the principal submatrix of $\bold{A}$ formed by restricting to rows
and columns whose indices are elements of the set $I.$
For any set $\cal{E},$ we write $conv(\cal{E})$ to denote the convex hull of the
set $\cal{E}$. For a closed convex cone ${\cal K}$, the generalized completely positive cone over ${\cal K}$ is defined as the set of symmetric matrices that are representable as the sum of rank one matrices of the form:
\begin{align*}
\displaystyle {\cal C}({\cal K}) = \{\bold{A} \in \mathcal{S}^n : \exists \bold{b}_1,\ldots,\bold{b}_p \in {\cal K} \mbox{ such that } \bold{A} = \sum_{k \in [p]} \bold{b}_k\bold{b}_k'\}.
\end{align*}
For ${\cal K} = \mathbb{R}^{n}_{+}$, ${\cal C}(\mathbb{R}^{n}_{+})$ is the cone of completely positive matrices. The dual to this cone is the cone of copositive matrices denoted as ${\cal C}^*(\mathbb{R}^{n}_{+})$. More generally for ${\cal K} = \mathbb{R}^{n} \times \mathbb{R}^{n}_{+}$, ${\cal C}(\mathbb{R}^{n} \times \mathbb{R}^{n}_{+})$ is given by
\begin{align}
\displaystyle {\cal C}(\mathbb{R}^{n} \times \mathbb{R}^{n}_{+}) = \left\{\begin{bmatrix}
\bold{A} & \bold{B}' \\\bold{B} & \bold{C}
\end{bmatrix} \in \mathcal{S}_+^{n \times n}: \bold{C} \in {\cal C}(\mathbb{R}^{n}_{+})\right\}.
\end{align}

\section{Literature review}
\label{sec:lit-review}
 There is now a fairly significant literature on methods that either compute the tight distributionally robust bound $Z^*$ or weaker upper bounds on $Z^*$ for mixed integer linear optimization problems \cite{Delage2010,Bertsimas2010,Natarajan2011,Kong2013,natarajan2017reduced,XuB18,Grani,Bertsimas2018}. In general, one of the difficulties that arises in exact formulations to compute $Z^*$ under moment-based ambiguity sets is that it involves optimization over the cone of completely positive matrices, which is typically intractable. This naturally leads to the question of identifying specific instances for which the problem is tractable, which is our focus in this paper. We review some of the key concepts briefly next, before discussing the contributions in this work.

\subsection{Exact Reformulations: Completely Positive Matrices and Quadratic Forms}
Problem (\ref{Moment-Problem}) for $R = 1$ corresponds to the case where the mean  $\boldsymbol{\mu} \in \mathbb{R}^n$ and the entire second moment matrix $\bold{\Pi} \in \mathcal{S}_+^{n}$ is specified. The distributionally robust bound studied in \cite{Natarajan2011} is:
\begin{align}
Z_{\text{full}}^*(\boldsymbol{\mu},\bold{\Pi}) =  \sup \left\lbrace
  \mathbb{E}_{\theta} \left[ Z(\bold{\tilde{c}})\right] \ : \
\begin{array}{ll}
  \mathbb{E}_{\theta}[\bold{\tilde{c}}] =\boldsymbol{ \mu}, \ \mathbb{E}_{\theta} [\bold{\tilde{c}}\bold{\tilde{c}}']  = \bold{\Pi}, \   \theta \in \mathcal{P}(\mathbb{R}^n)
\end{array}
\right\rbrace.
\label{Full-Moment-Problem}
\end{align}
An exact reformulation of the problem is obtained in \cite{Natarajan2011} by using the expected value of the following random variables as decision variables:
\begin{align*}
\mathbb{E}\left( \begin{bmatrix}
1 \\ \tilde{\bold{c}} \\ \bold{x(\tilde{c})}
\end{bmatrix}\begin{bmatrix}
1 \\ \tilde{\bold{c}} \\ \bold{x(\tilde{c})}
\end{bmatrix}'\right),
\end{align*}
where $ \bold{x(\tilde{c})}$ is a randomly chosen optimal solution for the objective coefficients $\tilde{\bold{c}}$. For the case when the decision variables in the set ${\cal I}$ in $\mathcal{X}$ are binary, building on the seminal work in \cite{Burer2009}, \cite{Natarajan2011} provided an equivalent reformulation of this problem, under mild assumptions on the set $\mathcal{X}$ as a generalized completely positive program of the form:
\begin{equation}
\label{opt:karthik-comp-pos-formulation}
 \begin{array}{rrlllll}
{Z}_{\text{full}}^*(\boldsymbol{\mu},\bold{\Pi}) &= \displaystyle \max_{\bold{p}, \bold{X},
  \bold{Y}}
& & trace(\bold{Y}) \\
& \mbox{s.t} & &
\begin{bmatrix}
 1 & \boldsymbol{\mu}' & \bold{p}'  \\
\boldsymbol{\mu} & \bold{\Pi} & \bold{Y}'\\
   \bold{p}  & \bold{Y} & \bold{X}\\
\end{bmatrix}  \in {\cal C}(\mathbb{R}_{+} \times \mathbb{R}^{n} \times \mathbb{R}^{n}_{+}),
\\
& & & \displaystyle \bold{a}_k'\bold{p} = b_k, & \forall k \in [p],\\
& & & \displaystyle \bold{a}_k'\bold{X}\bold{a}_k = b_k^2, & \forall k \in [p],\\
& & & \displaystyle X_{jj} = x_j, & \forall j \in {\cal I},
\end{array}
\end{equation}
where $\bold{a}_k'$ is the $k$th row of the matrix $\bold{A}$. Unfortunately, this problem is hard to solve due to the difficulty in characterizing the generalized completely positive cone ${\cal C}(\mathbb{R}_{+} \times \mathbb{R}^{n} \times \mathbb{R}^{n}_{+})$. 
For matrices of size $n \geq 5$, testing for membership in the  completely positive cone ${\cal C}(\mathbb{R}^{n}_{+})$ is known to be NP-hard \cite{Dickinson2014}. However, for $n \leq 4$, the completely positive cone of matrices coincides with the doubly nonnegative cone of matrices  $\mathcal{DNN}^n$ = $\mathcal{S}^n_+ \cap \mathcal{N}^n $ where $\mathcal{N}^n$ denote the set of matrices of size $n$ with nonnegative elements. It is straightforward to characterise the doubly nonnegative cone of matrices using psd and nonnegativity conditions and this provides a tractable relaxation to the completely positive cone, since ${\cal C}(\mathbb{R}^{n}_{+}) \subseteq \mathcal{DNN}^n$ for all $n$. The doubly nonnegative relaxation thus results in an upper bound on $Z^*$, which might not be tight. There are several hierarchies of psd and nonnegative cones that have been developed to generate tighter approximations of the completely positive cone and the dual copositive cone including the works of \cite{BomzeCDL17}, \cite{Bomze2002}, \cite{Zuluaga2005}, \cite{ParrilloThesis2000}. We note that completely positive and copositive programming representations of distributionally robust optimization problems under alternative ambiguity sets such as Wasserstein-based ambiguity sets have been recently developed in \cite{Grani} and \cite{XuB18}.


A related formulation that builds on characterizing the convex hull of quadratic forms over the feasible region and semidefinite optimization was proposed in \cite{natarajan2017reduced}. They established an equivalent tight formulation to compute $Z^*$ as follows:
\begin{equation}
  \label{opt:karthik-reduced-sdp-formulation}
 \begin{array}{rrllllll}
\displaystyle {Z}_{\text{full}}^*(\boldsymbol{\mu},\bold{\Pi}) &= \displaystyle \max_{\bold{p}, \bold{X},
  \bold{Y}}
& & trace(\bold{Y}) \\
& \mbox{s.t} & &
\begin{bmatrix}
 1 & \boldsymbol{\mu}' & \bold{p}'  \\
\boldsymbol{\mu} & \bold{\Pi} & \bold{Y}'\\
   \bold{p}  & \bold{Y} & \bold{X}\\
\end{bmatrix} \succeq 0,
\\
& & & \begin{pmatrix}
\bold{p}, \bold{X}
\end{pmatrix}\in conv \left\lbrace
 \begin{pmatrix}
\bold{x},\, \bold{x} {\bold{x}}'\end{pmatrix} : \bold{x} \in \mathcal{X}
\right \rbrace.
\end{array}
\end{equation}
 This exact formulation requires an explicit characterization of the convex hull of quadratic forms on the feasible region. Characterising this convex hull  is known to be NP-hard for sets such as $\mathcal{X} = \{0,1\}^n$ which corresponds to characterizing the Boolean quadric polytope (see \cite{Pitowsky1991}, \cite{Padberg1989}). However, the approach allows for the possibility of using valid inequalities that have been developed in deterministic instances for the Boolean quadric polytope, to develop tighter formulations for distributionally robust bounds in applications such as in the newsvendor problem (see \cite{natarajan2017reduced}). Efficient representations of the convex hull in (\ref{opt:karthik-reduced-sdp-formulation}) are known for some special cases of $\mathcal{X}$ in low dimensions as discussed in \cite{Anstreicher2010} and in some special cases, in higher dimensions as discussed in \cite{Burer2015,Burer2018}. Identifying instances where this set is efficiently representable remains an active area of research.

\subsection{Contributions}
Our contributions in the paper are the following:
 \begin{enumerate}
 \item In Section \ref{sec:blk-diag-results}, we study MILPs with random objective coefficients where the first moments are entirely known and only partial information of the second moments is provided, corresponding to non-overlapping subsets of ${\cal N}$. We provide a tight reformulation of the problem in the spirit of formulation (\ref{opt:karthik-reduced-sdp-formulation}), building on the results in \cite{natarajan2017reduced}. However, as we show, this formulation requires psd constraints on smaller matrices and furthermore, it involves characterizing a suitable projection of the convex hull of the set $\{(\bold{x}, \bold{x}\bold{x}'): \bold{x} \in \mathcal{X}\}$, rather than the full convex hull. This provides a reduced SDP formulation for the problem under block-diagonal patterns of covariance information.

 \item We provide an application of the formulation to appointment scheduling in \Cref{sec:app-sched}. In the distributionally robust appointment scheduling problem with moment-based ambiguity sets, tight polynomial-sized formulations exist only for the mean-variance setting which corresponds to $R = n$, to the best of our knowledge. On the other hand, with a full covariance matrix which corresponds to $R = 1$, this problem is known to be hard to solve. By identifying an efficient characterization of projection of the convex hull of the set $\{(\bold{x}, \bold{x}\bold{x}'): \bold{x} \in \mathcal{X}\}$ in this example, we identify a new polynomial-time solvable instance of distributionally robust appointment scheduling with partial correlation information when $R = 2$. We also identify polynomial-time solvable instances in the longest path problem on PERT networks and the assignment problem with random coefficients in \Cref{sec:path} and \Cref{sec:asg}.
 \item In \Cref{sec:computations}, we perform a detailed computational study of the proposed reformulation in the distributionally robust appointment scheduling application. We compare the results with alternative formulations and help identify specific structures of correlations where the new formulation is most valuable.  Finally we study the optimal schedules generated by various formulations including ours.
 \end{enumerate}
\section{Tight bounds in the presence of block-diagonal correlation information}
\label{sec:blk-diag-results}
\subsection{A reduced semidefinite program}
\label{sec:reduced-SDP}
In Theorem \ref{thm:main} below, we identify a reduced semidefinite
programming formulation for evaluating $Z^\ast$ in which the positive
semidefinite constraints are imposed only on smaller matrices of dimensions
$n_1,\ldots,n_r$, instead of a larger matrix of dimension $n$. Moreover, Theorem \ref{thm:main} asserts that it is sufficient
to enforce the $(n^2 + 3n)/2$ dimensional convex hull constraint (ignoring symmetry) in \eqref{opt:karthik-reduced-sdp-formulation} on a suitable selection
involving only $\sum_r (n_r^2 + 3n_r)/2$ variables.


\begin{theorem}
\label{thm:main}
Define $Z^*$ as the tight bound:
\begin{align*}
    Z^* = \sup \left\lbrace
  \mathbb{E}_\theta\left[\max_{\bold{x} \in \mathcal{X}}
  \tilde{\bold{c}}'\bold{x}\right]: \
  \mathbb{E}_{\theta}[\bold{\tilde{c}}] =\boldsymbol{\mu}, \
  \mathbb{E}_{\theta} [\tilde{\bold{c}}^r(\tilde{\bold{c}}^r)']  =
          \bold{\Pi}^r \text{ for } r \in [R\,], \ \theta \in
          \mathcal{P}(\mathbb{R}^n) \right\}.
\end{align*}
Define $\hat{Z}^*$ as the optimal objective value of the following
semidefinite program:
\begin{equation}
\label{opt:SmallSDP}
 \begin{array}{rrlllll}
\displaystyle \hat{Z}^* &= \displaystyle  \max_{\bold{p}, \bold{X}^r, \bold{Y}^r} & \displaystyle \sum_{r=1}^R trace(\bold{Y}^r) \\
& \mbox{s.t} & \displaystyle
\begin{bmatrix}
 1 & {\boldsymbol{\mu}^r}' & {\bold{p}^r}'  \\
\boldsymbol{\mu}^r & \bold{\Pi}^r & {\bold{Y}^r}'\\
   \bold{p}^r  & \bold{Y}^r& \bold{X}^r\\
\end{bmatrix} \succeq 0, \;\quad\quad  \text{ for } r \in
[R\,],
\\
& & \begin{pmatrix}
\bold{p}, \, \bold{X}^1,  \ldots, \, \bold{X}^R
\end{pmatrix}\in conv \left\lbrace
 \begin{pmatrix}
\bold{x},\, \bold{x}^1 {\bold{x}^1}',\ldots, \, \bold{x}^R {\bold{x}^R}'
\end{pmatrix} : \bold{x} \in \mathcal{X}
\right \rbrace.
\end{array}
\end{equation}
Then, $\hat{Z}^* = Z^*$.
\end{theorem}

Before proving this result, which forms the main part of this section, we discuss some implications. In comparison to formulation \eqref{opt:karthik-reduced-sdp-formulation}, formulation \eqref{opt:SmallSDP} involves psd constraints on multiple but much smaller matrices when $\max_r n_r$ is smaller than $n$. Furthermore, the theorem implies that only relevant projections of the convex hull of quadratic forms require to be characterised to compute $Z^*$, under block-diagonal correlation information. Such sparse characterizations have been previously exploited to identify polynomial-time solvable instances of unconstrained quadratic 0-1 optimization problems using an appropriate projection of the Boolean quadric polytope (see \cite{Padberg1989}). As we shall see in Section \ref{sec:poly-time}, the new formulation allows us
to derive compact representations that results in polynomial-time solvable instances for
the distributionally robust appointment scheduling problem, 
as well as for computing
worst-case bounds in PERT networks and bounds for the linear assignment problem with random objective.

\subsection{On chordal graphs and psd completion}
\label{sec:chordal-graphs}
A key element in the proof of Theorem \ref{thm:main} comprises in guaranteeing the existence of a psd matrix whose entries are partially
specified. Therefore, as a preparation towards the proof of Theorem
\ref{thm:main}, we provide a brief review of results on the psd
completion problem and a closely related notion of chordal graphs that
are relevant for our proofs; see, for example,
\cite{GRONE1984109,Laurent2009} and references therein for a detailed
exposition on the psd completion problem.

We call a matrix whose entries are specified only on a subset of its
positions as a \textit{partial matrix}. Suppose that $\bold{A}$ is a partial
matrix. The set of positions corresponding to the specified entries of
$\bold{A}$ is known as the \textit{pattern} of $\bold{A}$. A \textit{completion} of
the partial matrix $\bold{A}$ is simply a specification of the unspecified
entries of $\bold{A}$. If $\bold{A}$ is a partial symmetric matrix (that is, the
entry $A_{ji}$ is specified and is equal to $A_{ij}$ whenever $A_{ij}$
is specified) such that every principal specified submatrix of $\bold{A}$ is
psd, then $\bold{A}$ is said to be \textit{partial psd}. A \textit{psd
  completion} of the partial psd matrix $\bold{A}$ is said to exist if there
exists a specification of the unspecified entries of $\bold{A}$ such that the
fully specified matrix is psd.

A graph is said to be \textit{chordal} if any cycle of length greater
than or equal to four has a chord (see \cite{Berge1967}). Here, a
\textit{chord} is simply an edge that is not part of the cycle but
connects two vertices of the cycle. As we shall note in Lemma
\ref{lem:pe-ordering} below, the existence of a perfect elimination
ordering characterizes the chordal property of a graph. For a graph with $|V|$ vertices, a
\textit{perfect elimination ordering} is an ordering
$\beta_1, \ldots, \beta_{|V|}$ of the vertices such that for every
$i \in \{1,\ldots,\vert V \vert -1\},$ the set of vertices
$\{\beta_{i+1}, \beta_{i+2}, \ldots, \beta_{|V|}\} \cap
\mathcal{N}(\beta_i)$ form a clique, where $\mathcal{N}(v)$ is used to
denote the set of vertices adjacent to vertex
$v.$ 

The following well-known results on chordal graphs and psd completion
will be useful in proving Theorem \ref{thm:main}.
\begin{lemma}
  A graph is chordal if and only if it has a perfect elimination
  ordering.
  \label{lem:pe-ordering}
\end{lemma}
\begin{lemma}
  Every partial positive semidefinite matrix with pattern denoted by a graph $G$ (where the vertices denote the rows (or columns) of the matrix and an edge is present between two vertices if the corresponding entry is specified) has a
  positive semidefinite completion if and only if $G$ is a chordal
  graph.
  \label{lem:psd-comp-chordal}
\end{lemma}
The proofs of Lemma \ref{lem:pe-ordering} and
\ref{lem:psd-comp-chordal} can be found, respectively, in
Theorem 1, \cite{Rose1970} and Proposition 1 and Theorem 7, \cite{GRONE1984109}.

\subsection{Proof of Theorem \ref{thm:main}}
\label{sec:proof-main-thm}
\textbf{Step 1: To show $Z^\ast \leq \hat{Z}^\ast $.} It follows from
the definitions of $Z^\ast$ and
$Z^\ast_{full}(\boldsymbol{\mu},\bold{\Pi})$ in \eqref{Moment-Problem}
and \eqref{Full-Moment-Problem} that
  $Z^\ast = \max\{ Z^\ast_{\text{full}}(\boldsymbol{\mu},
  \bold{\Delta}):\bold{\Delta}  \in S_n^+,
  \bold{\Delta}[\mathcal{N}_r] =  \bold{\Pi}^r \text{ for } r \in [R]\}.$
  Therefore we have from \cite[Theorem 2]{natarajan2017reduced} (see
  formulation \eqref{opt:karthik-reduced-sdp-formulation}) that
\begin{equation}
  \label{opt:KarthikReducedSDP}
 \begin{array}{rrlllll}
  Z^\ast &= \displaystyle \max_{\bold{p}, \bold{X},
  \bold{Y}, \bold{\Delta}}
& & trace(\bold{Y}) \\
& \quad\ \text{s.t} & &
\begin{bmatrix}
 1 & \boldsymbol{\mu}' & \bold{p}'  \\
\boldsymbol{\mu} & \bold{\Delta} & \bold{Y}'\\
   \bold{p}  & \bold{Y}& \bold{X}\\
\end{bmatrix} \succeq 0,
\\
& & & \ \bold{\Delta}[\mathcal{N}_r] =  \bold{\Pi}^r, \quad \text{ for }
r \in [R\,],\\
& & &
\begin{pmatrix}
\bold{p}, & \bold{X}
\end{pmatrix}\in conv \left\lbrace
\begin{pmatrix} \bold{x}, & \bold{x}\bold{x}'
\end{pmatrix} : \bold{x} \in \mathcal{X}
\right\rbrace.\\
\end{array}
\end{equation}
Consider any $\bold{p}, \bold{X}, \bold{Y}, \bold{\Delta}$ feasible
for \eqref{opt:KarthikReducedSDP}. Take
$\bold{X}^r = \bold{X}[\mathcal{N}_r]$ and
$\bold{Y}^r = \bold{Y}[\mathcal{N}_r].$ The psd constraint in
\eqref{opt:KarthikReducedSDP} forces all the principal submatrices to
be psd. Given the block-diagonal partition, define $\{\mathcal{V}_r : r \in [R]\}$ to be the following subsets of
$\{1,\ldots,2n+1\}:$
\begin{align}
  \mathcal{V}_r = \{1\} \cup \{ i + 1: i \in \mathcal{N}_r\} \cup \{ n + i + 1: i
  \in \mathcal{N}_r\}, \quad \text{ for } r \in [R].
\label{Defn-Vr-Sets}
\end{align} Then, the principal submatrices formed by restricting
to entries from the index set $\mathcal{V}_r,$ for $r \in [R],$
satisfy,
\begin{align*}
\begin{bmatrix}
 1 & \boldsymbol{\mu}' & \bold{p}'  \\
\boldsymbol{\mu} & \bold{\Delta} & \bold{Y}'\\
   \bold{p}  & \bold{Y}& \bold{X}\\
\end{bmatrix}[\mathcal{V}_r]  =
\begin{bmatrix}
 1 & {\boldsymbol{\mu}^r}' & {\bold{p}^r}'  \\
\boldsymbol{\mu}^r & \bold{\Delta}[\mathcal{N}_r]&
\bold{Y}[\mathcal{N}_r]'\\
   \bold{p}^r  & \bold{Y}[\mathcal{N}_r]& \bold{X}[\mathcal{N}_r]\\
\end{bmatrix} =
\begin{bmatrix}
 1 & {\boldsymbol{\mu}^r}' & {\bold{p}^r}'  \\
\boldsymbol{\mu}^r & \bold{\Pi}^r & {\bold{Y}^r}'\\
   \bold{p}^r  & \bold{Y}^r & \bold{X}^r\\
\end{bmatrix}
\succeq 0.
\end{align*}
In addition, since
$(\bold{p}, \bold{X}) \in conv\{(\bold{x}, \bold{x}\bold{x}^\prime): \bold{x}
\in \mathcal{X}\},$ it is immediate that the principal submatrices
$\bold{X}[\mathcal{N}_r] = \bold{X}^r$ satisfy the projected convex
hull constraint in \eqref{opt:SmallSDP}.  Furthermore, the objective,
$trace(\bold{Y}) = \sum_{r} trace(\bold{Y}[\mathcal{N}_r]) =
\sum_{r} trace(\bold{Y}^r).$ Thus for every
$\bold{p}, \bold{X}, \bold{Y}, \bold{\Delta}$ feasible for
\eqref{opt:KarthikReducedSDP}, there exist
$\{\bold{p}, \bold{X}^r, \bold{Y}^r: r \in [R\,]\}$ feasible for
\eqref{opt:SmallSDP} with the same objective. Therefore
$Z^* \leq \hat{Z}^\ast.$

\textbf{\\Step 2: To show $Z^\ast \geq \hat{Z}^\ast$}\\
Suppose that
$\{\bold{{p}}_\ast,\bold{{ X}}_\ast^r,\bold{{Y}}_\ast^r: r \in
[R]\}$ maximizes \eqref{opt:SmallSDP}.  We show that
$Z^\ast \geq \hat{Z}^\ast$ by constructing $\hat{\bold{p}}, \hat{\bold{X}}, \hat{\bold{Y}}, \hat{\bold{\Delta}}$
feasible to \eqref{opt:KarthikReducedSDP} and
$trace(\hat{\bold{Y}}) = \hat{Z}^\ast = \sum_{r}
trace(\bold{Y}_\ast^r).$

\noindent\textit{Construction of $\hat{\bold{p}}$}: Simply, take $\hat{\bold{p}} = \bold{p}_\ast.$

\noindent\textit{Construction of $\hat{\bold{X}}$}: It follows from
Carath\'{e}odory's theorem and the convex hull constraint in
\eqref{opt:SmallSDP} that there exists $\hat{\mathcal{X}},$ a subset
of $\mathcal{X},$ containing at most $1+\sum_r (n_r^2 + 3n_r)/2$
elements such that,
\begin{align*}
  \left(
 \hat{\bold{p}} ,\bold{{X}}^1_*, \ldots, \bold{{X}}^R_*
  \right) = \sum_{\bold{x} \in \hat{\mathcal{X}}}\alpha_\bold{x}
  \left(\bold{x},\ \bold{x}^1{\bold{x}^1}', \ldots,\bold{x}^R
  {\bold{x}^R}' \right),
\end{align*}
for some $\{\alpha_{\bold{x}}: \bold{x} \in \hat{\mathcal{X}}\}$ satisfying
$\alpha_\bold{x}\geq 0,$
$\sum_{\bold{x} \in \hat{\mathcal{X}}} \alpha_{\bold{x}} = 1.$ Now
take
$\hat{\bold{X}} = \sum_{\bold{x} \in \hat{\mathcal{X}}} \alpha_{\bold{x}}
\bold{x}\bold{x}'.$ 
Then,
\begin{align}
  \begin{bmatrix}
1 & \hat{\bold{p}} \\
\hat{\bold{p}} &\bold{\hat{X}}
\end{bmatrix} =  \sum\limits_{\bold{x} \in \hat{\mathcal{X}}}
                   \alpha_{\bold{x}} \begin{bmatrix}
1 \\
\bold{x}
\end{bmatrix}
\begin{bmatrix}
1 \\
\bold{x}
\end{bmatrix}'
\quad \text{ and } \quad \hat{\bold{X}}[\mathcal{N}_r] = \sum_{\bold{x} \in \hat{\mathcal{X}}}
\alpha_{\bold{x}} \bold{x}^r{\bold{x}^r}' = \bold{X}_\ast^r, \text{ for } r \in [R].
\label{eq:inter-chull-feasibility}
\end{align}

\noindent\textit{Construction of $\hat{\bold{Y}}$ and $\hat{\bold{\Delta}}$}:
Consider $n \times n$ partial matrices $\hat{\bold{Y}}_p$ and
$\hat{\bold{\Delta}}_p$ with entries specified only along the
following principal submatrices:
\begin{align}
  \hat{\bold{Y}}_p[\mathcal{N}_r] = \bold{Y}_\ast^r \quad \text{ and } \quad
  \hat{\bold{\Delta}}_p[\mathcal{N}_r] = \bold{\Pi}^r, \quad \text{ for
  } r \in  [R\,].
\label{eq:inter-submatrix-entires}
\end{align}
Next, consider a $(2n+1) \times (2n+1)$ partial symmetric matrix
$\bold{L}_p$ constructed in terms of the partial matrices
$\hat{\bold{Y}}_p,$ $\hat{\bold{\Delta}}_p$ and the fully specified
matrix $\hat{\bold{X}}$ as follows:
\begin{align*}
 \bold{L}_p =
  \begin{bmatrix}
    1 & \boldsymbol{\mu}' & \hat{\bold{p}}'  \\
    \boldsymbol{\mu} & \hat{\bold{\Delta}}_p & \hat{\bold{Y}}_p\\
    \hat{\bold{p}}  & {\hat{\bold{Y}}_p}'& \hat{\bold{X}}\\
  \end{bmatrix}
=
\begin{bmatrix}
  1 & {\boldsymbol{\mu}^1}' & \ldots &  {\boldsymbol{\mu}^R}' & {\bold{p}_\ast^1}' & \ldots &{\bold{p}^R_\ast}' \\
  \boldsymbol{\mu}^1 & \bold{\Pi}^1 & ? & ? &  {\bold{{Y}}_\ast^1}' & ? & ? \\
  \vdots & ?  &  \ddots & ? & ? & \ddots  & ? \\
  \boldsymbol{\mu}^R & ?  & ?  &  \bold{\Pi}^R &  ? & ? & {\bold{{Y}}_\ast^R}'\\
 \bold{p}_\ast^1 &   \bold{{Y}}_\ast^1 & ? & ? &  &   &  \\
  \vdots & ? & \ddots  & ?  & & \hat{\bold{X}} &\\
  \bold{p}_\ast^R & ? & ? &  \bold{{Y}}_\ast^R &  &  &  \\
\end{bmatrix}.
\end{align*}
The entries marked `?' denote missing entries. 
By demonstrating that the underlying pattern of $\bold{L}_p$ is
chordal, Lemma \ref{lemma:psd-completion} below establishes that there
exists a psd completion for the partial matrix $\bold{L}_p.$


\begin{lemma}
\label{lemma:psd-completion}
The matrix $\bold{L}_p$ has a completion $\bold{L}_{comp}$ such that
$\bold{L}_{comp} \succeq 0 $.
\end{lemma}

\begin{proof}
  Consider the following construction of an undirected graph $G$ with
  vertex set,
  $V=\left\lbrace s, c_1, c_2, \ldots, c_n, x_1, \ldots,x_n
  \right\rbrace,$ comprising $2n+1$ vertices. To define the edge set,
  identify the vertices $s, c_1, c_2, \ldots, c_n, x_1, \ldots, x_n,$
  respectively, with the rows (or columns) numbered $1,2,\ldots,2n+1$ of
  the partial matrix $\bold{L}_p.$ We assign an edge between two
  vertices of $G$ only if the the respective entry of the partial
  matrix $\bold{L}_p$ is specified. Therefore, graph $G$ represents
  the pattern of the partial matrix $\bold{L}_p$

  With the above described construction of graph $G,$ note that the
  vertices $\{x_1, \ldots, x_n \}$ form a clique in $G$ as the matrix
  $\hat{\bold{X}}$ is specified completely. The edges between the vertices
  $\{c_1,\ldots,c_n\}$ 
  correspond to the specified entries of the partial matrix
  $\hat{\bold{\Delta}}_p$. Likewise, the edges between vertices
  $\{c_1,\ldots,c_n\}$ and vertices $\{x_1,\ldots,x_n\}$ correspond to
  the known entries of the partial martial $\hat{\bold{Y}}_p$. Thus,
  for any $r \in [R\,],$ when restricted to vertices corresponding
  to $\bold{c}^r$ and $\bold{x}^r$, we again have a clique (see
  \Cref{fig:chordal-graph} for an illustration).

  Next, consider the ordering of the vertices,
  $c_1, c_2, \ldots,c_n, s, x_1,x_2,\ldots,x_n,$ of $G.$ Since the
  vertices $\{s,x_1,\ldots,x_n\}$ form a clique, it is immediate that
  for any $x_i,$ the neighbors of the node that appear after it in the
  ordering also form a clique. The same reasoning applies for the vertex
  $s.$ For any $i \in [n],$ let $r_i$ be the unique $r \in [R\,]$ such
  that $i \in \mathcal{N}_{r_i}.$ Subsequently, the neighbors of $c_i$
  that appear after it in the ordering comprises the collection $\{c_j, s, x_k: j,k \in \mathcal{N}_{r_i}, j > i\},$ which again
  forms a clique. This is because the vertices
  $\{s,c_j, x_j: j \in \mathcal{N}_r\}$ form a clique, for any
  $r \in [R\,].$ Consequently, the ordering
  $c_1, c_2, \ldots,c_n, s, x_1,x_2,\ldots,x_n$ is a perfect
  elimination ordering for the graph $G.$ Then due to Lemma
  \ref{lem:pe-ordering}, $G$ is a chordal graph.

  Recalling the definition of $\mathcal{V}_r$ in \eqref{Defn-Vr-Sets},
  observe that any fully specified principal submatrix of $\bold{L}_p$
  is a principal submatrix of
  \begin{align*}
    \bold{L}_p[\mathcal{V}_r] =
        \begin{bmatrix}
      1 & {\boldsymbol{\mu}^r}' & (\hat{\bold{p}}^r)'\\
      \boldsymbol{\mu}^r & \hat{\bold{\Delta}}_p[\mathcal{N}_r] & \hat{\bold{Y}}_p[\mathcal{N}_r]'\\
      \hat{\bold{p}}^r &\hat{\bold{Y}}_p[\mathcal{N}_r] &
      \hat{\bold{X}}[\mathcal{N}_r]
    \end{bmatrix}
                                          =
    \begin{bmatrix}
      1 & {\boldsymbol{\mu}^r}' & {\bold{p}_\ast^r}'\\
      \boldsymbol{\mu}^r & \bold{\Pi}^r & {\bold{Y}_\ast^r}'\\
      \bold{p}_\ast^r & \bold{Y}_\ast^r &  \bold{X}_\ast^r
    \end{bmatrix},
  \end{align*}
  for some $r \in [R\,].$ The latter equality follows from
  \eqref{eq:inter-submatrix-entires} and the second observation in
  \eqref{eq:inter-chull-feasibility}.  Since
  $\bold{p}_\ast, \bold{Y}^r_\ast, \bold{X}^r_\ast$ are taken to be
  feasible for \eqref{opt:SmallSDP}, we have that
  $\bold{L}_p[\mathcal{V}_r] \succeq 0$ for any $r \in [R\,].$ With
  the `maximal' fully specified principal submatrices
  $\{\bold{L}_p[\mathcal{V}_r]: r \in [R\,]\}$ being psd, we
  have that all the fully specified principal submatrices are
  psd. Therefore $\bold{L}_p$ is partial psd.

  Finally, with the pattern underlying the partial psd matrix
  $\bold{L}_p$ forming a chordal graph, it follows from Lemma
  \ref{lem:psd-comp-chordal} that there exists a psd completion for
  $\bold{L}_p.$ 
  \hfill$\Box$
\end{proof}

To complete the proof, consider the psd completion $\bold{L}_{comp}$ of $\bold{L}_p.$ Take
$\hat{\bold{\Delta}} := \bold{L}_{comp}[\{2,\ldots,n+1\}]$ and
$\hat{\bold{Y}}$ to be the $n \times n$ submatrix of $\bold{L}_{comp}$
formed from entries in rows $\{2,\ldots,n+1\}$ and columns
$\{n+2,\ldots, 2n+1\}.$ Then $\bold{L}_{comp} \succeq 0$ allows us to
write,
\begin{align}
  \bold{L}_{comp}  =
  \begin{bmatrix}
 1 & \boldsymbol{\mu}' & \hat{\bold{p}}'  \\
\boldsymbol{\mu} & \hat{\bold{\Delta}} & \hat{\bold{Y}}'\\
   \hat{\bold{p}}  & \hat{\bold{Y}}& \hat{\bold{X}}\\
\end{bmatrix} \ \succeq \  0.
\label{eq:inter-psd-feasibility}
\end{align}
Since the specified entries of $\bold{L}_p$ match with that of
$\bold{L}_{comp},$ it follows from the construction of $\bold{L}_p$
that $\hat{\bold{Y}}$ is a completion of $\hat{\bold{Y}}_p$ and
$\hat{\bold{\Delta}}$ is a psd completion of $\hat{\bold{\Delta}}_p;$
the latter completion is psd because the principal submatrices of
$\bold{L}_{comp}$ are psd.  Therefore, we have from
\eqref{eq:inter-submatrix-entires} that
\begin{align}
  \label{eq:inter-Delta-submatrix-eq}
  \hat{\bold{\Delta}}[\mathcal{N}_r] = \bold{\Pi}^r \quad \text{ and }
  \quad \hat{\bold{Y}}[\mathcal{N}_r] = \bold{Y}_\ast^r.
\end{align}
Furthermore, as we have taken
$\{\bold{p}_\ast, \bold{X}_\ast^r, \bold{Y}_\ast^r: r \in [R]\}$ to
maximize \eqref{opt:SmallSDP}, we have,
\begin{align}
  \hat{Z}^\ast = \sum_{r = 1}^R trace(\bold{Y}_\ast^r) = \sum_{r =
  1}^R trace(\hat{\bold{Y}}[\mathcal{N}_r])  = trace(\hat{\bold{Y}}).
\label{eq:inter-optimal-obj}
\end{align}

It follows from \eqref{eq:inter-psd-feasibility} and the first of the
two equations in \eqref{eq:inter-chull-feasibility} and
\eqref{eq:inter-Delta-submatrix-eq} that
$\hat{\bold{p}}, \hat{\bold{\Delta}}, \hat{\bold{X}}, \hat{\bold{Y}}$
are feasible for \eqref{opt:KarthikReducedSDP}. Therefore, the optimal
value of \eqref{opt:KarthikReducedSDP}, denoted by $Z^\ast,$ satisfies
${Z}^\ast \geq trace(\hat{\bold{Y}}).$ The desired
${Z}^\ast \geq \hat{Z}^\ast$ is now a consequence of
\eqref{eq:inter-optimal-obj}. This completes Step 2 and the proof of
Theorem \ref{thm:main}. \hfill$\Box$

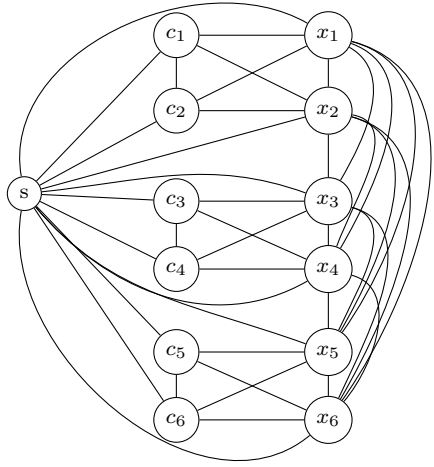
\begin{figure}[h]
\centering
\begin{tikzpicture}
\tikzset{vertex/.style = {shape=circle,draw,minimum size=0.5em}}
\tikzset{edge/.style = {-,> = latex'}}
\node[vertex] (s) at  (0,0) {s};
\node[vertex] (c1) at  (2,2.1) {$c_1$};
\node[vertex] (x1) at  (4,2.1) {$x_1$};
\node[vertex] (c2) at  (2,1.1) {$c_2$};
\node[vertex] (x2) at  (4,1.1) {$x_2$};
\node[vertex] (c3) at  (2,-0.1) {$c_3$};
\node[vertex] (x3) at  (4,-0.1) {$x_3$};
\node[vertex] (c4) at  (2,-1) {$c_4$};
\node[vertex] (x4) at  (4,-1) {$x_4$};
\node[vertex] (c5) at  (2,-2.1) {$\tiny c_{5}$};
\node[vertex] (x5) at  (4,-2.1) {$\tiny x_{5}$};
\node[vertex] (c6) at  (2,-3) {$c_{6}$};
\node[vertex] (x6) at  (4,-3) {$x_{6}$};
\draw[edge] (s) to (c1) ;\draw[edge] (s) to (c2) ; \draw[edge] (s) to[out=100,in=150] (x1) ;
\draw[edge] (s) to (x2) ;
\draw[edge] (c1) to (x1) ;
\draw[edge] (c2) to (x2) ;
\draw[edge] (c1) to (c2) ;
\draw[edge] (x1) to (x2) ;
\draw[edge] (c1) to (x2) ;
\draw[edge] (c2) to (x1) ;
\draw[edge] (s) to (c3) ;\draw[edge] (s) to (c4) ; \draw[edge] (s) to[out=6,in=160] (x3) ;
\draw[edge] (s) to [out=-50,in=-150] (x4) ;
\draw[edge] (c3) to (x3) ;
\draw[edge] (c4) to (x4) ;
\draw[edge] (c3) to (c4) ;
\draw[edge] (x3) to (x4) ;
\draw[edge] (c3) to (x4) ;
\draw[edge] (c4) to (x3) ;
\draw[edge] (x2) to (x3) ;
\draw[edge] (x1) to[out=-30,in=60] (x3) ;
\draw[edge] (x1) to[out=-20,in=60] (x4) ;
\draw[edge] (x2) to[out=-10,in=70] (x4) ;
\draw[edge] (s) to (c5) ;\draw[edge] (s) to (c6) ; \draw[edge] (s) to[out=-50,in=160] (x5) ;
\draw[edge] (s) to [out=-100,in=-140] (x6) ;
\draw[edge] (c5) to (x5) ;
\draw[edge] (c6) to (x6) ;
\draw[edge] (c5) to (c6) ;
\draw[edge] (x5) to (x6) ;
\draw[edge] (c5) to (x6) ;
\draw[edge] (c6) to (x5) ;
\draw[edge] (x4) to (x5) ;
\draw[edge] (x1) to[out=-15,in=55] (x5) ;
\draw[edge] (x2) to[out=-15,in=60] (x5) ;
\draw[edge] (x3) to[out=-15,in=65] (x5) ;
\draw[edge] (x1) to[out=-15,in=55] (x6) ;
\draw[edge] (x2) to[out=-15,in=60] (x6) ;
\draw[edge] (x3) to[out=-15,in=65] (x6) ;
\draw[edge] (x4) to[out=-15,in=50] (x6) ;
\end{tikzpicture}
\caption{Illustration for the graph $G$ for the case where $n=6$ and
  the partition is given by
  $\mathcal{N}_1 = \{1,2\}, \mathcal{N}_2 = \{3,4\}$ and
  $\mathcal{N}_3 = \{5,6\}.$ For $k$ odd, $c_{k-1}$ and $c_k$ are not
  connected and $\left\lbrace c_k, x_k, c_{k+1}, x_{k+1},
    s\right\rbrace$ form a
  clique.  }
\label{fig:chordal-graph}
\end{figure}

\subsection{On the structure of a worst-case distribution}
\label{sec:gen-worst-case-distr}
In this section, we exhibit a probability distribution for
$\bold{\tilde{c}}$ that attains the optimal value $Z^*$ of
\eqref{opt:SmallSDP}. The construction is along the lines of the worst
case distribution proposed in \cite{natarajan2017reduced}, adapted to our model.

We begin with a result on psd matrix factorization in
\cite{natarajan2017reduced}. The following definition of Moore-Penrose
pseudoinverse (see \cite{penrose_1955,rao1972}) is useful in stating
the psd matrix factorization in Lemma \ref{lemma:decomposition}. Let
$\bold{X}$ be a matrix of dimension $k_1 \times k_2.$ Then the
Moore-Penrose pseudoinverse of $\bold{X}$ is a matrix
$\bold{X}^\dagger$ of dimension $k_2 \times k_1$ and is defined as a
unique solution to the set of four equations:\\
\[\bold{X}\bold{X}^\dagger\bold{X} = \bold{X}, \quad
\bold{X}^\dagger \bold{X} \bold{X}^\dagger = \bold{X}^\dagger, \quad
\bold{X}\bold{X}^\dagger = (\bold{X}\bold{X}^\dagger)', \quad \text{
  and } \quad
\bold{X}^\dagger \bold{X} = (\bold{X}^\dagger \bold{X})'.\]



\begin{lemma}\cite[Theorem 1]{natarajan2017reduced}
\label{lemma:decomposition}
Suppose that $\bold{L}$ is a $(k_1 + k_2) \times (k_1+k_2)$ positive
semidefinite block matrix of the form,
\begin{align}
\bold{L}=\begin{bmatrix}
\bold{A} & \bold{B}' \\\bold{B} & \bold{C}
\end{bmatrix} \succeq 0,
\label{eq:block-matrix-form}
\end{align}
 where the matrices
$\bold{A} \in \mathbb{R}^{k_1 \times k_1}, \bold{C} \in
\mathbb{R}^{k_2 \times k_2}$ are symmetric and the matrix $\bold{C}$
admits an explicit factorization given by
$\bold{C} = \bold{V}\bold{V}'.$ Then $\bold{L}$ admits the following
factorization:
\begin{equation}
\bold{L} = \begin{bmatrix}
\bold{B'(V^{\dagger})'} \\\bold{V}
\end{bmatrix}
\begin{bmatrix}
\bold{B'(V^{\dagger})'} \\\bold{V}
\end{bmatrix}' + \begin{bmatrix}
\bold{U} \\ 0
\end{bmatrix}\begin{bmatrix}
\bold{U} \\ 0
\end{bmatrix}',
\end{equation}
where the matrix $\bold{U}$ is defined such that $\bold{A} -
\bold{B'C^{\dagger}B} = \bold{U U'} \succeq 0.$
\end{lemma}

For a given partition $\{\mathcal{N}_r: r \in [R\,]\}$ and
projected covariance matrices $\{\bold{\Pi}^r: r \in [R\,]\},$ suppose that
$\{\bold{p}_\ast, \bold{X}_\ast^r, \bold{Y}^r_\ast : r \in [R]\}$
maximizes \eqref{opt:SmallSDP}.  As in the proof of Theorem
\ref{thm:main}, it follows from Carath\'{e}odory's theorem and the
convex hull constraint in \eqref{opt:SmallSDP} that there exists
$\hat{\mathcal{X}},$ a subset of $\mathcal{X},$ containing at most
$1+\sum_r (n_r^2 + 3n_r)/2$ elements such that,
\begin{align*}
  \left(
  \bold{p}_\ast ,\bold{{X}}^1_*, \ldots, \bold{{X}}^R_*
  \right) = \sum_{\bold{x} \in \hat{\mathcal{X}}}\alpha_\bold{x}
  \left(\bold{x},\ \bold{x}^1{\bold{x}^1}', \ldots,\bold{x}^R
  {\bold{x}^R}' \right),
\end{align*}
for some $\{\alpha_{\bold{x}}: \bold{x} \in \hat{\mathcal{X}}\}$
satisfying $\alpha_\bold{x}\geq 0,$
$\sum_{\bold{x} \in \hat{\mathcal{X}}} \alpha_{\bold{x}} = 1.$
Consequently, for any $r \in [R\,],$ we have from Lemma
\ref{lemma:decomposition} that,
\begin{align}
\begin{bmatrix}
  \bold{\Pi}^r & {\boldsymbol{\mu}^r} & {\bold{Y}_\ast^r}' \\
  {\boldsymbol{\mu}^r}' & 1 & {\bold{p}_\ast^r}'\\
  \bold{Y}_\ast^r & \bold{p}_\ast^r & \bold{X}_\ast^r
\end{bmatrix}
&= \sum\limits_{\bold{x} \in \hat{\mathcal{X}}} \alpha_{\bold{x}}
\begin{bmatrix}
  \bold{\Pi}^r &  {\boldsymbol{\mu}^r} & {\bold{Y}_\ast^r}' \\
  {\boldsymbol{\mu}^r}' & 1 & {\bold{x}^r}'\\
  \bold{Y}_\ast^r & \ \bold{x}^r & \ \bold{x}^r {\bold{x}^r}'
\end{bmatrix} \nonumber\\
& = \sum\limits_{\bold{x} \in \hat{\mathcal{X}}} \alpha_{\bold{x}}
\begin{bmatrix}
 \bold{d}_r(\bold{x}^r)\\1\\\bold{x}^r
\end{bmatrix}
\begin{bmatrix}
 \bold{d}_r(\bold{x}^r)\\1\\\bold{x}^r
\end{bmatrix}'   +
\begin{bmatrix}
\bold{\Phi}_{r} & \bold{0}_{n_r,1} & \bold{0}_{n_r,n_r} \\
 \bold{0}_{1,n_r} & 0 &  \bold{0}_{1,n_r} \\
 \bold{0}_{n_r,n_r}& 0 &  \bold{0}_{n_r,n_r} \\
\end{bmatrix},
\label{eq:inter-psd-factorization}
\end{align}
where $\bold{d}_r(\bold{x}^r)\in \mathbb{R}^{n_r} $ and
$\bold{\Phi}_r \in \mathcal{S}_{n_r}^+$ for every $r \in [R\,].$ From the
above factorization, observe that,
\begin{align}
  trace(\bold{Y}^r_\ast) = trace\left(\sum_{x \in \hat{\mathcal{X}}}
  \alpha_{\bold{x}}\bold{x}^r\bold{d}_r(\bold{x}^r)'\right) =
  \sum_{x \in \hat{\mathcal{X}}}
  \alpha_{\bold{x}} \sum_{r = 1}^R \bold{d}_r(\bold{x}^r)'
  {\bold{x}^r}.
\label{eq:wc-objective}
\end{align}

For completeness, we will now list explicitly the expressions for the means $\bold{d}_r(\bold{x}^r)$ and $ \bold{\Phi_r}$. These expressions are obtained by making appropriate substitutions as per \Cref{lemma:decomposition}.
For every $r \in [R]$, define the matrix $\bold{V}_r$ of size $(n_r + 1) \times m_r$ where $m_r $ is the number of points in the projected space $\mathcal{X}^r$ as  follows.

\begin{align}
\bold{V}_r = \begin{bmatrix}
\ldots & \sqrt{\alpha_r(\bold{x}^r)} & \ldots\\
\vdots &  \sqrt{\alpha_r(\bold{x}^r) }\bold{x}^r& \vdots
\end{bmatrix}
\end{align} Each column of $\bold{V}_r$ corresponds to an element $\bold{x}^r$ of $\mathcal{X}^r$ and is of the form $ \begin{bmatrix}\sqrt{\alpha_r(\bold{x}^r)} \\  \sqrt{\alpha_r(\bold{x}^r) }\bold{x}^r \end{bmatrix}$.
Define $\bold{\Phi}_r$ of size $n_r \times n_r$ as:
\begin{align}
\bold{\Phi}_r = \bold{\Pi}^r - [\boldsymbol{\mu}^r \quad \hat{\bold{Y}}_*^r]\begin{bmatrix}
1 & \sum_{\bold{x}^r} \alpha_r(\bold{x}^r) (\bold{x}^{r})'\\
\sum_{\bold{x}^r}\alpha_r(\bold{x}^r) \bold{x}^{r} & \sum_{\bold{x}^r}\alpha_r(\bold{x}^r) \bold{x}^{r} (\bold{x}^{r})'
\end{bmatrix}^{\dagger} \begin{bmatrix}
\boldsymbol{\mu}^r \\
\hat{\bold{Y}}_*^r
\end{bmatrix}
\end{align}
The mean vector $\bold{d}_{r}(\bold{x}_r)$ is set to be the column vector of  the matrix \\
$\begin{bmatrix}
\boldsymbol{\mu}^r & \hat{\bold{Y}}_*^r
\end{bmatrix}(\bold{V}_r^{\dagger})'\times 1/\sqrt{\alpha_{r}(\bold{x}^r})$ corresponding to where $\bold{x}_r$ occurs in $\bold{V}_r$.

\begin{proposition}
  Suppose that
  $\{\bold{p}_\ast, \bold{X}_\ast^r, \bold{Y}^r_\ast : r \in [R]\}$ maximizes \eqref{opt:SmallSDP} and that there exists a
  finite $\hat{\mathcal{X}} \subseteq \mathcal{X}$ and
  $\{\alpha_{\bold{x}}: \bold{x} \in \hat{\mathcal{X}}\}$ satisfying
  \eqref{eq:inter-psd-factorization}. Let $\theta^\ast$ be the
  distribution of $\tilde{\bold{c}}$ generated as follows:
\begin{itemize}
\item[] Step 1: Generate a random vector
  $\tilde{\bold{x}} \in \hat{\mathcal{X}} \subseteq \mathcal{X}$ such that
  $P(\tilde{\bold{x}} = \bold{x}) = \alpha_{\bold{x}}.$
\item[] Step 2: For every $r \in [R\,],$ independently generate a normally
  distributed random vector $\tilde{\bold{z}}_r \in \mathbb{R}^{n_r},$
  conditionally on ${\bold{x}},$ with mean $ \bold{d}_r(\bold{x}^r)$ and
  covariance $\bold{\Phi_r}$. Set $\bold{\tilde{c}}^r = \tilde{\bold{z}}_r.$
\end{itemize}
Then $\theta^\ast$ attains the maximum in \eqref{Moment-Problem}. 
\label{prop:wc-dist}
\end{proposition}
\begin{proof}
  Consider $(\tilde{\bold{x}} ,\tilde{\bold{c}})$ generated jointly according
  to the described steps. Then it follows from the law of iterated
  expectations that,
  $\mathbb{E}[f(\tilde{\bold{c}})] = \mathbb{E}[
  \mathbb{E}[f(\tilde{\bold{c}}) \vert \tilde{\bold{x}} ] = \sum_{\bold{x} \in
    \hat{\mathcal{X}}}\alpha_{\bold{x}}\mathbb{E}[f(\tilde{\bold{c}})
  \vert \tilde{\bold{x}}  = \bold{x}],$ for any function $f.$ As a result, we have
  from \eqref{eq:inter-psd-factorization} that for any $r \in [R\,],$
\begin{align}
  \mathbb{E}
\begin{bmatrix}
\bold{\tilde{c}}^r \\ 1
\end{bmatrix}\begin{bmatrix}
\bold{\tilde{c}}^r \\ 1
\end{bmatrix}'
 &= \sum_{\bold{x} \in \hat{\mathcal{X}}} \alpha_ {\bold{x}}
  \begin{bmatrix}
\bold{d}_r(\bold{x}^r)\bold{d}_r(\bold{x}^r)' + \bold{\Phi}_r &  \ \bold{d}_r({\bold{x}^r})\\
{\bold{d}_r(\bold{x}^r)}' & 1
 \end{bmatrix}  = \begin{bmatrix}
 \bold{\Pi}^r & \boldsymbol{\mu}^r\\
 {\boldsymbol{\mu}^r}' & 1
 \end{bmatrix}.
\label{eq:marg-feasibility}
\end{align}
Moreover, as $\tilde{\bold{x}}\in \mathcal{X},$ the objective
$\mathbb{E}[\max_{\bold{x} \in \mathcal{X}}\bold{\tilde{c}}'\bold{x}]$
satisfies,
\begin{align*}
Z^\ast  &\geq  \mathbb{E}\left[\max_{\bold{x} \in
  \mathcal{X}}\bold{\tilde{c}}'\bold{x} \right]
  \geq \mathbb{E}[\tilde{\bold{c}}'\tilde{\bold{x}}]
  = \mathbb{E}\left[ \mathbb{E}\left[\tilde{\bold{c}} \, \vert\,
    \tilde{\bold{x}} \right]' \tilde{\bold{x}}\right]  = \mathbb{E}\left[ \sum_{r=1}^R
    \mathbb{E}\left[\tilde{\bold{c}}^r \, \vert\,   \tilde{\bold{x}} \right]' \tilde{\bold{x}}^r \right] \\
  &= \mathbb{E}\left[ \sum_{r=1}^R   \bold{d}_r(\tilde{\bold{x}}^r)' \tilde{\bold{x}}^r  \right]
  = \sum_{\bold{x} \in \mathcal{X}} \alpha_{\bold{x}} \sum_{r=1}^{R}
    \bold{d}_{r}(\bold{x}^r)'\bold{x}^r = \sum_{r=1}^{R}
    trace(\hat{\bold{Y}}^r_*) = \hat{Z}^\ast = Z^\ast,
\end{align*}
where the last three equalities follow, respectively, from
\eqref{eq:wc-objective}, the optimality of
$\{\bold{p}_\ast, \bold{X}_\ast^r, \bold{Y}_\ast^r: r \in [R\,]\}$ for
\eqref{opt:SmallSDP}, and Theorem \ref{thm:main}. Combining this
observation with \eqref{eq:marg-feasibility}, we have that the
distribution of $\tilde{\bold{c}},$ denoted by $\theta^\ast,$
is feasible and it attains the maximum in \eqref{Moment-Problem}.
\hfill$\Box$
\end{proof}

\section{Polynomial-time solvable instances}
\label{sec:poly-time}
In this section, we identify efficient representations of the convex
hull constraint in \eqref{opt:SmallSDP} for three illustrative
applications. The common theme in these applications is that the
derived efficient characterizations, in turn, result in
polynomial-time solvable instances for the partial covariance based
distributionally robust formulation in \eqref{Moment-Problem}. As far
as we know, the example we consider in Section \ref{sec:app-sched} is the first example of such an
approach towards appointment scheduling that results in a
polynomial-time solvable tight reformulation in the presence of explicitly known correlation information between uncertain processing times of the jobs
to be scheduled.

\subsection{Appointment scheduling}
\label{sec:app-sched}

\subsubsection{Problem description}
In the presence of uncertainty in the processing durations of jobs
for a sequence of customers, the appointment scheduling problem
deals with identifying customer reporting times that minimize the
total amount of time spent by customers waiting for service after
arrival. As an example, consider $n$ patients who need to meet a
doctor. Let $\tilde{u}_i$ be the random service duration of patient
$i \in [n].$ Suppose that all patients arrive exactly at the
reporting time allotted to them.  If we let $s_i$ denote the duration
scheduled for patient $i,$ then the reporting time for patient $i$ is
$\sum_{j=1}^{i-1} s_j.$ 
We take the waiting time of the first patient to be zero. Then the
waiting time for patient $i,$ denoted by $w_i,$ satisfies the
well-known Lindley's recursion for the waiting time in single-server
queues:
\begin{equation}
  w_1 = 0, \quad w_i = \max(w_{i-1} + \tilde{u}_{i-1} - s_{i-1},0), \quad i =
  2,\ldots,n.
\label{eq:lindley-recursion}
\end{equation}

The total waiting time for all patients is the sum of waiting times
$\sum_{i=1}^{n} w_i$. The overtime of the doctor can be modeled as $w_{n+1} = \max(w_n + \tilde{u}_n - s_n, 0)$. Then the total waiting
time of the patients and the overtime of the doctor are cumulatively
captured by,
\begin{align}
\label{opt:det-large}
  f(\tilde{\bold{u}}, \bold{s}) &=
                                         \sum_{i=1}^n\max(w_i +
                                         \tilde{u}_i-s_i,
                                         0). 
\end{align}
Define
$S = \{ \bold{s} \in \mathbb{R}_+^n: s_1 + \ldots + s_n \leq T\},$
where $T$ is a positive upper time limit within which the schedules
should be fit.  It is then natural to seek a schedule sequence
$\bold{s} \in S$ that minimizes
$\mathbb{E}[f(\tilde{\bold{u}}, \bold{s})].$

The described setup is applicable to schedule appointments in various
situations where a single server processes the arriving jobs on a
first-come-first-serve basis. In settings where the jobs to be
processed are dependent and the joint distribution of their processing
times $\tilde{\bold{u}}$ is difficult to be fully specified, an
approach that has gained much attention over the last decade is to
seek distributionally robust schedules that minimize the worst
case waiting time,
$\sup_{\theta \in \mathcal{P}}
E_\theta[f(\tilde{\bold{u}},\bold{s})];$ here, the set $\mathcal{P}$
is taken to be the family of all probability distributions consistent
with the information known about the probability distribution of
$\tilde{\bold{u}}$. This problem was first studied in \cite{Kong2013} where complete information on the first moment $\boldsymbol{\mu}$ and  second moment matrix $\bold{\Pi}$ is assumed to be available on the service time durations. Building on the completely positive formulation in (\ref{opt:karthik-comp-pos-formulation}), the problem can be reformulated as:

\begin{equation}
\label{opt:cp-app-sched-formulation}
\begin{array}{rrlllll}
{Z}_{\text{app}}(\boldsymbol{\mu}, \bold{\Pi}, \bold{s})&= \displaystyle \max_{\bold{p}, \bold{X},
  \bold{Y}}
& & trace(\bold{Y}) - \bold{s}'\bold{p} \\
& \mbox{s.t} & &
\begin{bmatrix}
 1 & \boldsymbol{\mu}' & \bold{p}'  \\
\boldsymbol{\mu} & \bold{\Pi} & \bold{Y}'\\
   \bold{p}  & \bold{Y} & \bold{X}\\
\end{bmatrix}  \in {\cal C}(\mathbb{R}_{+} \times \mathbb{R}^{n} \times \mathbb{R}^{n}_{+}),
\\
& & & \displaystyle [\bold{A}, -\bold{I}_n]\bold{p} = \bold{b}, \\
& & & \displaystyle [\bold{a}'_i, -\bold{e}'_i] \bold{X}[\bold{a}'_i, -\bold{e}'_i]' = b_i^2, & \forall i \in [n].
\end{array}
\end{equation}
In the above formulation $\bold{p} \in \mathbb{R}^{2n}$, $\bold{Y} \in \mathbb{R}^{n \times 2n}$ and $\bold{X} \in \mathbb{R}^{2n \times 2n}$. The matrix $\bold{A}  \in \mathbb{R}^{n\times n}$ is such that
$A_{jj}=-1 \text{ for } j \in [n] $,
$A_{j+1,j}=1 \text{ for } j \in [n-1] $,
$ b_j= -1 \text{ for } j \in [n]$ and
$\bold{a}_i'$ indicates row $i$ of $\bold{A}$ and $\bold{e}_i \in \mathbb{R}^n$ denotes a column vector with $1$ at position $i$ and zero everywhere else. The above formulation provides the distributionally robust bound for a given schedule $\bold{s}$. The distributionally robust schedules may be obtained by optimizing the dual of the above formulation over the dual variables as well as $\bold{s}$ as follows:
\begin{equation}
\label{opt:cp-app-sched-formulation-dual}
 \begin{array}{rrllll}
& \underset{\substack{\bold{s}, \alpha,  \boldsymbol{\beta}, \bold{\Gamma, \gamma},\\\bold{u}, \bold{v}, \bold{\Lambda}, \boldsymbol{\eta}, \boldsymbol{\chi}}}{\text{min}}
& & trace(\bold{\Pi}'\bold{\Gamma}) + \boldsymbol{\mu}'  \boldsymbol{\beta} + \alpha+ \gamma \\
& \text{s.t} & &
\begin{pmatrix}
 \alpha + \sum\limits_{i=1}^n v_i - u_i  & \frac{\boldsymbol{\beta'}}{2} & \frac{\begin{pmatrix}
 \bold{s} + 2\boldsymbol{\eta} \\
 2\boldsymbol{\chi} \end{pmatrix}' - \sum\limits_{i=1}^n u_i \begin{pmatrix} \bold{a}_i \\ \bold{e}_i \end{pmatrix}'}{2} \\
 \\
\frac{\boldsymbol{\beta}}{2} & \bold{\Gamma} &  \begin{pmatrix}
 -0.5\bold{I}_n\\ \bold{O}_n\end{pmatrix}' \\
 \frac{\begin{pmatrix}
 \bold{s} + 2\boldsymbol{\eta} \\
 2\boldsymbol{\chi} \end{pmatrix} - \sum\limits_{i=1}^n u_i \begin{pmatrix} \bold{a}_i \\ \bold{e}_i \end{pmatrix}}{2}  & \begin{pmatrix}
 -0.5\bold{I}_n\\ \bold{O}_n
 \end{pmatrix} & -\sum\limits_{i=1}^n v_i \begin{pmatrix}
 \bold{a}_i \\ -\bold{e}_i
\end{pmatrix}\begin{pmatrix}
  \bold{a}_i \\ -\bold{e}_i
\end{pmatrix}' + \bold{\Lambda}  \\
   \end{pmatrix} \succeq 0 \\
 & & &
   \begin{pmatrix}
   \gamma & \begin{pmatrix}
   \boldsymbol{\eta} \\
   \boldsymbol{\chi}
   \end{pmatrix}' \\
    \begin{pmatrix}
   \boldsymbol{\eta} \\
   \boldsymbol{\chi}
   \end{pmatrix}  & \bold{\Lambda}  \end{pmatrix} \in  {\cal C^*}(\mathbb{R}_+^{2n+1})
 \end{array}
 \end{equation}
In the above formulation $\alpha, \gamma \in \mathbb{R}$, $\boldsymbol{\beta}, \bold{u}, \bold{v}, \boldsymbol{\eta}, \boldsymbol{\chi} \in \mathbb{R}^n$, $\bold{\Gamma} \in \mathbb{R}^{n\times n}$, $\bold{\Lambda} \in \mathbb{R}^{2n \times 2n}$. $\bold{I}_n$ and $\bold{O}_n$ denote the identity matrix and the square matrix of zeros respectively, both of size $n$.

\subsubsection{Polynomial-time solvable instance}
To illustrate the applicability of Theorem \ref{thm:main} in this
context, suppose that the number of patients, $n,$ is even without loss
of generality, and the mean of service times $\tilde{\bold{u}}$ is
fully specified, and only the entries,
$\{\Pi_{ii}, \Pi_{j,j+1}, \Pi_{j+1,j}: i \in [n], j \in
\{1,3,\ldots,n-1\}\}$ of the second moment matrix,
$\bold{\Pi} = [\Pi_{ij}],$ are specified. This corresponds to knowing the correlations among service time durations of adjoining patients. Recalling the definition of
the partition $\{\mathcal{N}_r: r \in [R\,]\}$ of $[n],$ this partial
specification of the second moments corresponds to the
scenario where,
\begin{align}
  R = n/2,  \quad \text{ and } \quad   \mathcal{N}_r = \{2r-1, 2r\}, \text{ for
  } r =  1,\ldots, R.
\label{eq:partition-app-sched}
\end{align}
For any given schedule
$\bold{s} \in S,$ consider the worst-case expected total waiting time,
\begin{align}
  \label{eq:app-sched-mom-prob}
  Z^\ast_{app}(\bold{s}) = \sup\big\{
  \mathbb{E}_\theta\left[f(\tilde{\bold{u}},\bold{s})\right] : \
  &\mathbb{E}_\theta\left[\tilde{u}_i
    \right] =\mu_i, \mathbb{E}_\theta[\tilde{u}_i^2] =\Pi_{ii}, \text{ for }
    i \in [n], \nonumber\\
  & \mathbb{E}_\theta[\tilde{u}_j \tilde{u}_{j+1}]
    =\Pi_{j,j+1}, \text{  for } j \in \{1,3,...n-1\} \big\}.
\end{align}

Our key result is that by an appropriate application of Theorem~\ref{thm:main}, we obtain a polynomial-time
solvable formulation
for evaluating $Z^\ast_{app}(\bold{s})$ in Theorem
\ref{thm:SmallSDP-app-sched}.

\begin{theorem}
 \label{thm:SmallSDP-app-sched}
 Given a schedule $\bold{s} \in S,$ suppose that
 $Z^\ast_{app}(\bold{s})$ is defined as in
 \eqref{eq:app-sched-mom-prob}. Then,
 \begin{equation*}
\label{opt:SmallSDP-app-sched}
 \begin{array}{rrllll}
   Z^*_{app}(\bold{s}) =& \displaystyle \max_{p_i, X_{ij}, Y_{ij}, t_{kj}} & \displaystyle
   \sum_{i=1}^n (Y_{ii} - s_ip_i)\\
   &\mbox{s.t.} &
\displaystyle\begin{bmatrix}
 1 & \mu_i & \mu_{i+1} & p_i &p_{i+1}   \\
 \mu_i  & \Pi_{ii} &   \Pi_{i,i+1}& Y_{ii} & Y_{i,i+1} \\
\mu_{i+1} & \Pi_{i,i+1} &  \Pi_{i+1,i+1}& Y_{i+1,i} &  Y_{i+1,i+1} \\
  p_i & Y_{ii}&  Y_{i+1,i} & X_{ii} & X_{i,i+1}\\
     p_{i+1} & Y_{i,i+1}& Y_{i+1,i+1} & X_{i,i+1} & X_{i+1,i+1}\\
   \end{bmatrix} \succeq 0, \quad \text{ for } i \text{ odd, } i \in [n],  \\
   & & p_i = \sum\limits_{k=1}^i \sum\limits_{j=i}^{n+1}
   t_{kj} (j-i), \quad \text{ for } i \in [n],\\
   & &X_{ii} = \sum\limits_{k=1}^i \sum\limits_{j=i}^{n+1}
   t_{kj} (j-i)^2,  \quad \text{ for }  i \in [n],\\
   & &X_{i,i+1} = X_{i+1,i} = \sum\limits_{k=1}^i
   \sum\limits_{j={i+1}}^{n+1} t_{kj} (j-i)(j-(i+1)), \text{ for }
i \text{ odd, } i \in [n],\\
   & &\sum\limits_{k=1}^i \sum\limits_{j=i}^{n+1} t_{kj}
   = 1, \quad \text{ for } i \in [n], \\
   & & t_{kj} \geq 0, \quad \text{ for } 1 \leq k \leq j
   \leq n+1.
\end{array}
\end{equation*}
 \end{theorem}

 The proof of Theorem \ref{thm:SmallSDP-app-sched} is presented in
 Section \ref{sec:proof-thm-app-sched}. As demonstrated in Corollary
 \ref{cor-schedules-dual} below, an optimal schedule that minimizes the
 worst-case total expected waiting time can be obtained by considering
 the dual minimization problem of the semidefinite program in Theorem
 \ref{thm:SmallSDP-app-sched}.

\begin{corollary}
  Given $T > 0,$ a schedule
  $s \in S = \{ \bold{s} \in \mathbb{R}_+^n: s_1 + \ldots + s_n
  \leq T\}$ that minimizes $Z_{app}^\ast(\bold{s})$ can be obtained by
  solving the following semidefinite program:
  \begin{equation*}
\label{opt:optimal_schedule}
 \begin{aligned}
   Z_{app}^\ast &= \min\limits_{\bold{s}, \boldsymbol{\eta},
     \boldsymbol{\beta}, \bold{\Gamma},\boldsymbol{\rho},
     \boldsymbol{\delta}, \boldsymbol{\tau}, \boldsymbol{\gamma}}
   \sum_{\substack{i \in [n],\\i \,\text{odd}}}\eta_i +
   \sum_{i=1}^n\beta_i\mu_i + \sum_{\substack{i \in [n],\\i \,\text{odd}}}
   \sum_{k,l \in \{i,i+1\}} \hspace{-10pt}\Gamma_{kl}\Pi_{kl} + \sum_{i=1}^{n+1} \rho_i \\
   &\text{s.t }\quad
   \begin{bmatrix}
 2\eta_i & {\beta_i} & {\beta_{i+1}}& {\delta_i + s_i} & {\delta_{i+1} + s_{i+1}} \\
  {\beta_i} &  2\Gamma_{ii} & {\Gamma_{i,i+1}} & -{1} & 0 \\
    {\beta_{i+1}} & {\Gamma_{i,i+1}} & 2\Gamma_{i+1,i+1} & 0 & -{1} \\
    {\delta_i + s_i} & -{1} & 0  & 2\gamma_i & {\tau_i} \\
    {\delta_{i+1} + s_{i+1}} & 0   & -{1} & {\tau_i} & 2\gamma_{i+1} \\
  \end{bmatrix} \succeq 0, \quad \text{ for } i \text{ odd}, i \in
  [n],\\
  &\quad\quad \sum_{i=k}^j \rho_i \geq \sum_{i=k}^{\min\{j,n\}}
  \hspace{-10pt}\delta_i (j-i) +
  \sum_{i=k}^{\min\{j,n\}}\hspace{-10pt} \gamma_i (j-i)^2 +
  \sum\limits_{\substack{i=k \\ i \text{
        odd}}}^{\min\{j,n\}}\hspace{-10pt} \tau_i
  (j-i) (j- (i+1)),\\
  &\hspace{220pt} \text{ for }1 \leq k \leq j \leq n+1,\\
  &\quad\quad \sum_{i=1}^n s_i \leq T, \quad s_i \geq 0, \text{ for }
  i \in [n].
 \end{aligned}
 \end{equation*}
\label{cor-schedules-dual}
\end{corollary}

\begin{proof}
  The result follows by performing a joint minimization over
  $\bold{s} \in S$ and the objective of the dual of the semidefinite
  program in Theorem \ref{thm:SmallSDP-app-sched}. This is because,
  for any $\bold{s} \in S,$ the value of the semidefinite program in
  Theorem \ref{thm:SmallSDP-app-sched} is equal to that of its dual
  minimization problem. Indeed, the existence of an interior feasible
  point for the dual problem can be exhibited as follows. Given
  $\bold{s} \in S,$ set all the variables other than
  $\eta_i, \gamma_i,\gamma_{i+1},\Gamma_{ii}, \Gamma_{i+1,i+1},$ for
  $i$ odd, to zero, and let $\Gamma_{ii}\gamma_i > 1/4,$
  $\rho_i > \gamma_i(n+1-i)^2,$ for every $i \in [n];$ fix $\eta_i$,
  for $i$ odd, to be arbitrarily positive; this assignment results in
  a dual feasible solution where none of the constraints are active.
  Moreover, the requirement that
  $\bold{\Pi}^r - \boldsymbol{\mu}^r {\boldsymbol{\mu}^r}' \succ 0,$
  for every $r \in [R\,]$ is sufficient to guarantee strong duality.
\end{proof}

\subsubsection{A proof of Theorem \ref{thm:SmallSDP-app-sched}}
\label{sec:proof-thm-app-sched}
\textbf{Step 1: Recasting the waiting time
$f(\tilde{\bold{u}},\bold{s})$ in  the form of \eqref{eq:inner-obj}.}
Given a fixed sequence of schedules $\bold{s} = (s_1,\ldots,s_n),$ 
the recursive structure in \eqref{eq:lindley-recursion} allows writing
the total waiting time, $f(\tilde{\bold{u}},\bold{s}),$ as the value
of the following linear program:
\begin{equation*}
\begin{array}{rrlllll}
  & \displaystyle \min_{\bold{w}} & & \displaystyle \sum_{i=1}^{n+1} w_i \\
  & \displaystyle \mbox{s.t} & & \displaystyle w_i \geq w_{i-1} + \tilde{u}_{i-1} - s_{i-1},
  &\text{for } i= 2, \ldots, n+1,\\
  & & & w_i \geq 0, & \displaystyle \text{for } i = 1,\ldots,n+1.
\end{array}
\end{equation*}
Define $\tilde{\bold{c}}(\bold{s}) := \tilde{\bold{u}} - \bold{s}.$ The dual of this linear
program results in,
\begin{equation}
\label{opt:det_app_sched}
\begin{array}{rrlllll}
  f(\tilde{\bold{u}},\bold{s})= &\displaystyle \max_{\bold{x}} &
  \tilde{\bold{c}}(\bold{s})'\bold{x} \\
  &\mbox{s.t.} & x_i - x_{i-1} \geq -1, &\text{for } i= 2, \ldots, n-1, \\
  & & x_n \leq 1, \\
  & & x_i \geq 0, &\text{for } i=1,\ldots,n,
\end{array}
\end{equation}

The constraints in \eqref{opt:det_app_sched} are such that any subset
of $n$ active constraints satisfy, for every $i \in [n],$
either $x_i = 0$ or $x_{i-1} = x_i + 1.$ It has been shown in
\cite{Zangwill1966,Zangwill1969} that any $\bold{x}= (x_1,\ldots,x_n)$
with this special structure can be uniquely represented as a partition
of intervals of integers in $\{ 1, \ldots, n+1\}.$ This structure was first used in \cite{Mak2015} to identify a tractable instance of appointment scheduling with mean-variance information and to the case with no-shows in \cite{Jiang2017}. Lemma
\ref{lem:extr-pts-X-app-sched} below exploits this representation to
characterize the extreme points of the feasible region to
\eqref{opt:det_app_sched}. For completeness, we provide the proof here.

\begin{lemma}
  The extreme points of the feasible region in
  \eqref{opt:det_app_sched} is given by,
\begin{align}
  \label{eq:app-sched-polytope}
  \mathcal{X}_{app} = \bigg\{\bold{x} \in \mathbb{R}_+^n:
  x_i =   \sum_{k=1}^i&\sum_{j=i}^{n+1} T_{kj}(j-i), \text{ for } i \in
    [n], \sum_{k=1}^i\sum_{j=i}^{n+1}T_{kj} = 1, \text{ for } i \in
    [n],   \nonumber\\
  &T_{kj} \in \{0,1\}, \text{ for } 1 \leq k \leq j
    \leq n+1 \bigg\}.
\end{align}
\label{lem:extr-pts-X-app-sched}
\end{lemma}
\begin{proof}
  Recall our observation on the constraints in
  \eqref{opt:det_app_sched} that any subset of $n$ active constraints
  must have that, for every $i \in [n],$ either $x_i = 0$ or
  $x_{i-1} = x_i + 1.$ Therefore, any $\bold{x}$ in the feasible
  region to \eqref{opt:det_app_sched} is an extreme point if and only
  if either $x_i = 0$ or $x_{i-1} = x_i + 1,$ for every $i \in [n].$


  Now, for an extreme point $\bold{x} = (x_1,\ldots,x_n),$ let
  $I_{\bold{x}}$ be the unique partition of intervals of integers
  $\{1,2,\ldots,n, n+1\}$ such that the interval
  $[k,j] := \{k, k+1, \ldots, j\},$ for $k \leq j,$ belongs to the
  partition $I_{\bold{x}}$ if and only if
  $x_j = 0, x_{j-1}=1,\ldots,x_k = j-k.$ Thus there exists a bijection
  between the extreme points of the feasible region to
  \eqref{opt:det_app_sched} and the collection of partitions of
  integer intervals of $\{1,2,\ldots,n+1\}.$ For illustration, if
  $n = 3$ and $\bold{x} = (0,0,0)$ then
  $I_{(0,0,0)} =\{ [1], [2] , [3], [4] \};$ likewise, the points
  $(3,2,1), (1,0,1),$ are identified with their respective partitions
  given by, $I_{(3,2,1)}= \{[1, 4]\}$ and
  $I_{(1,0,1)} = \{ [1,2], [3,4] \},$ and vice versa.

  Next, for any extreme point $\bold{x}$ (whose unique interval
  partition representation is $I_{\bold{x}}$), consider the following
  assignment of values to the variables
  $(T_{kj}: 1 \leq k \leq j \leq n +1):$
  \begin{align}
    T_{kj} =
    \begin{cases}
      1 \quad\quad&\text{ if the integer interval } [k,j] \in
      I_{\bold{x}},\\
      0 &\text{ otherwise}.
    \end{cases}
    \label{defn-Tkj}
  \end{align}
  It follows from the very construction of the interval partition
  $I_{\bold{x}}$ that that only one of $\{T_{kj}: k \leq i \leq j\}$
  equals 1, for every $i \in [n],$ and
  $x_i = \sum_{k=1}^i \sum_{j=i}^{n+1} T_{kj}(j-i).$ Therefore any
  extreme point of the feasible region of \eqref{opt:det_app_sched}
  lies in $\mathcal{X}_{app}.$

  \sloppy{On the other hand, for any $\bold{x} = (x_1,\ldots,x_n)$ in
  $\mathcal{X}_{app},$ we have
  $x_i = \sum_{k=1}^i \sum_{j=i}^{n+1} T_{kj}(j-i)$ satisfying,}
    \begin{align*}
      x_i =
      \begin{cases}
        0 \quad\quad &\text{ if } T_{ki} = 1 \text{ for
          some } k \leq i,\\
       1  &\text{ if } i = n \text{ and } T_{kn} = 0,\\
       x_{i+1} + 1 &\text{ otherwise,}
      \end{cases}
    \end{align*}
    for every $i \in [n].$ Here we have again used the observation
    that for any given assignment of variables $T_{kj} \in \{0,1\}$
    satisfying
    $\sum_{k=1}^i\sum_{j=i}^{n+1}T_{kj} = 1, \text{ for every } i \in
    [n],$ only one of $\{T_{kj}: k \leq i \leq j\}$ equals one. Since
    any $\bold{x} \in \mathcal{X}_{app}$ satisfies $x_i = 0$ or
    $x_{i-1} = x_i + 1$ for every $i \in [n],$ we arrive at the
    conclusion that $\mathcal{X}_{app}$ is indeed the set of extreme
    points of the feasible region to
    \eqref{opt:det_app_sched}. \hfill$\Box$
\end{proof}
As the feasible region to the linear program in
\eqref{opt:det_app_sched} is bounded, there exists an extreme point at
which attains the maximum is attained. Then as a consequence of Lemma
\ref{lem:extr-pts-X-app-sched}, we have that
\begin{align}
  \label{eq:step1-outcome}
  f(\tilde{\bold{u}},\bold{s}) = \max\left\{
  \tilde{\bold{c}}(\bold{s})'\bold{x}: \bold{x} \in
  \mathcal{X}_{app}\right\}.
\end{align}

\noindent \textbf{Step 2: Application of Theorem \ref{thm:main}.}  For the
partition $\{\mathcal{N}_r: r \in [R]\}$ specified in
\eqref{eq:partition-app-sched}, we use
\eqref{eq:step1-outcome} to express $Z_{app}^\ast(\bold{s})$ as,
\begin{align*}
  \sup\left\{ \mathbb{E}_{\theta}\left[
  \max_{\bold{x} \in \mathcal{X}_{app}} \tilde{\bold{c}}(\bold{s})'\bold{x}
  \right]:
  \mathbb{E}_\theta[\tilde{\bold{c}}(\bold{s})] =
  \boldsymbol{\mu} - \bold{s},
  \mathbb{E}_{\theta}\left[\tilde{\bold{c}}(\bold{s})^r
  {\tilde{\bold{c}}(\bold{s})^r}'\right] = \bold{\Pi}^r_{\bold{s}} \ \forall r \in [R\,], \theta
  \in \mathcal{P}(\mathbb{R}^n)
  \right\},
\end{align*}
where, for $i = 1,3,\ldots,n-1,$ the second moment of
$(\tilde{\bold{c}}_i(\bold{s}), \tilde{\bold{c}}_{i+1}(\bold{s}))$ is
specified by,
\begin{align*}
   \bold{\Pi}^{\lceil i /2 \rceil}_{\bold{s}} =
  \begin{bmatrix}
    \Pi_{ii} & \Pi_{i,i+1} \\
    \Pi_{i,i+1} & \Pi_{i+1,i+1}
  \end{bmatrix}
   - \begin{bmatrix}
      s_i \\ s_{i+1}
      \end{bmatrix}
      \begin{bmatrix}
      \mu_i \\ \mu_{i+1}
      \end{bmatrix}'
     - \begin{bmatrix}
      \mu_i \\ \mu_{i+1}
      \end{bmatrix}
      \begin{bmatrix}
       s_i \\ s_{i+1}
      \end{bmatrix}'
  +    \begin{bmatrix}
      s_i \\ s_{i+1}
      \end{bmatrix}
     \begin{bmatrix}
      s_i \\ s_{i+1}
      \end{bmatrix}'.
\end{align*}
Then as an application of Theorem \ref{thm:main}, we can write
$Z^\ast_{app}(\bold{s})$ as the value of the semidefinite program in
\eqref{opt:SmallSDP} by replacing parameters
$\boldsymbol{\mu}^r, \bold{\Pi}^r,$ respectively, with
$\boldsymbol{\mu}^r - \bold{s}^r$ and $\bold{\Pi}^r_{\bold{s}}.$
Further, changing the variables $\bold{Y}^r$ to
$\bold{Y}^r - \bold{p}^r{\bold{s}^r}'$ for $r \in [R\,],$ the
objective in \eqref{opt:SmallSDP} becomes,
$\sum_{r} trace(\bold{Y}^r - \bold{p}^r{\bold{s}^r}'),$ and the
psd constraints in \eqref{opt:SmallSDP} becomes,
\begin{align*}
\begin{bmatrix}
  1 & & & {\boldsymbol{\mu}^r}' - {\bold{s}^r}' & & & {\bold{p}^r}'\\
  \boldsymbol{\mu}^r - \bold{s}^r & & & \bold{\Pi}^r - \bold{s}^r{\boldsymbol{\mu}^r}'
  - {\boldsymbol{\mu}^r}{\bold{s}^r}' + {\bold{s}^r\bold{s}^r}' & & & {\bold{Y}^r}' -
  \bold{s}^r{\bold{p}^r}'\\
  \bold{p}^r & & & {\bold{Y}^r} -  {\bold{p}^r}{\bold{s}^r}' & & & \bold{X}^r
\end{bmatrix}
\succeq 0.
\end{align*}
This psd constraint is equivalently written as,
\begin{align*}
\begin{bmatrix}
  1  & {\boldsymbol{\mu}^r}' &  {\bold{p}^r}'\\
  \boldsymbol{\mu}^r  & \bold{\Pi}^r & {\bold{Y}^r}' \\
  \bold{p}^r &  {\bold{Y}^r}  & \bold{X}^r
\end{bmatrix}
\succeq 0,
\end{align*}
due to the identical constraints
that arise as a result of Schur complement  conditions (for psd matrices) on both  the constraints above. Indeed, block-matrices of the form \eqref{eq:block-matrix-form} are psd if and only if both $ \bold{A}$ and $\bold{C} - \bold{B}'\bold{A}^{-1}\bold{B}$ are psd;
for the block matrices in the above constraints, take $\bold{A} = 1$ to verify the desired equivalence.

With these observations, we have
\begin{equation}
  \label{eq:inter-small-sdp-app-sched}
  \begin{array}{rrlllllll}
     Z^*_{app}(\bold{s}) =& \displaystyle \max_{p_i, X_{ij}, Y_{ij}, t_{kj}} &
\displaystyle   \sum_{i=1}^n (Y_{ii} - s_ip_i)\\
   & \mbox{s.t.} &
\begin{bmatrix}
 1 & \mu_i & \mu_{i+1} & p_i &p_{i+1}   \\
 \mu_i  & \Pi_{ii} &   \Pi_{i,i+1}& Y_{ii} & Y_{i,i+1} \\
\mu_{i+1} & \Pi_{i,i+1} &  \Pi_{i+1,i+1}& Y_{i+1,i} &  Y_{i+1,i+1} \\
  p_i & Y_{ii}&  Y_{i+1,i} & X_{ii} & X_{i,i+1}\\
     p_{i+1} & Y_{i,i+1}& Y_{i+1,i+1} & X_{i,i+1} & X_{i+1,i+1}\\
   \end{bmatrix} \succeq 0, \quad \text{ for } i \text{ odd
   }, \\
& & \displaystyle \left(p_1,\ldots,p_n,X_{11},\ldots,X_{nn},X_{12},X_{34},\ldots,X_{n-1,n}\right)  \in C_{app},
\end{array}
\end{equation}
where
\begin{align*}
 C_{app} = conv\left\{ \left(x_1,   \ldots , x_n , x_1^2, \ldots , x_n^2,
x_1x_2, x_3x_4, \ldots ,x_{n-1}x_{n}\right): \bold{x} \in
  \mathcal{X}_{app}\right\}.
\end{align*}

\noindent \textbf{Step 3: An efficient representation of the convex hull
  $C_{app}.$} We now complete the proof of Theorem
\ref{thm:SmallSDP-app-sched} by identifying a characterization of the
convex hull $C_{app}$ that leads to an efficient representation of the
last constraint written in \eqref{eq:inter-small-sdp-app-sched}.
\begin{proposition}
  The set $C_{app}$ is equivalently written as,
  \begin{align*}
    C_{app} =
    \bigg\{&\left(p_1,\ldots,p_n,X_{11},\ldots,X_{nn},X_{12},X_{34},\ldots,X_{n-1,n}\right)
           \in \mathbb{R}^{5n/2} : \nonumber\\
         &\quad p_i =   \sum_{k=1}^i\sum_{j=i}^{n+1} t_{kj}(j-i),  \ X_{ii} =
           \sum_{k=1}^i\sum_{j=i}^{n+1} t_{kj}(j-i)^2,\ \text{ for } i \in
           [n], \nonumber\\
    &\quad X_{i,i+1} = \sum\limits_{k=1}^i \sum\limits_{j={i+1}}^{n+1}
      t_{kj} (j-i)(j-(i+1)), \ \text{ for } i \in [n], i \text{ odd}, \nonumber\\
    &\quad \sum_{k=1}^i\sum_{j=i}^{n+1}t_{kj} = 1, \text{ for } i \in
      [n], \quad
      t_{kj} \geq 0 \text{ for } 1 \leq k \leq j   \leq n+1 \bigg\}.
\end{align*}
\label{prop:capp-char}
\end{proposition}
\begin{proof}
  Take any $\bold{x} \in \mathcal{X}_{app}.$ It follows from the
  characterization in \eqref{eq:app-sched-polytope} that there exists
  an assignment for variables $T_{kj} \in \{0,1\}$ such that only one
  of $\{T_{kj}: k \leq i \leq j\}$ equals one, for every $i \in [n],$
  and $x_i = \sum_{k=1}^i\sum_{j=i}^{n+1}T_{kj}(j-i).$ Therefore,
  $x_i = j-i$ and $x_i^2 = (j-i)^2$ for the unique $j \geq i$ such
  that $T_{kj}=1.$ Equivalently, we have
  \begin{align}
    \label{eq:xi-sq}
    x_i^2 = \sum_{k=1}^i\sum_{j=i}^{n+1}T_{kj}(j-i)^2.
  \end{align}
  Again, since only one of $\{T_{kj}: k \leq i \leq j\}$ equals one, for
  every $i \in [n],$ we have,
   \begin{align}
 \label{eqn:T_product_zero}
     T_{kj} T_{ab} = 0,  \quad\text{when either } k \neq a \text{ or } j
     \neq b \text{ and }  k \leq i \leq j, a \leq i \leq b.
   \end{align}
   Equipped with this observation, consider:
    \begin{align}
    x_ix_{i+1} &=
    \left(\sum_{k=1}^i\sum_{j=i}^{n+1}T_{kj}(j-i)\right)\left(\sum_{a=1}^{i+1}\sum_{b=i+1}^{n+1}T_{ab}(j-(i+1))\right) \nonumber\\
    &=\sum\limits_{k=1}^i \sum\limits_{j=i}^{n+1} T_{kj}(j-i)   \times
      \sum\limits_{a=1}^{i} \sum\limits_{b=i+1}^{n+1} T_{ab}(b- (i+1))
      \nonumber\\
      &\quad\quad \quad\quad+ \sum\limits_{k=1}^i \sum\limits_{j=i}^{n+1} T_{kj}(j-i)
      \times \sum\limits_{b=i+1}^{n+1} T_{i+1,b}(b- (i+1)),
       \nonumber
  \end{align}
  where the latter summand is equal to zero because, a) the terms for
  which $j = i$ are zero due to the appearance of $j-i,$ (see \Cref{fig:case1}) and b) the
  terms for which $j > i$ are zero due to the appearance of
  $T_{kj} T_{i+1,b},$ which is zero due to \eqref{eqn:T_product_zero}  (illustrated in  \Cref{fig:case2}) .
  Likewise, in the first summand, the terms for which
  $k \neq a, j \neq b$ vanish due to \eqref{eqn:T_product_zero}.  As a
  result,
   \begin{align}
     x_ix_{i+1} = \sum\limits_{k=1}^i \sum\limits_{j=i+1}^{n+1} T_{kj}
     (j-i) (j- (i+1)).
     \label{eq:cross-terms-representation}
   \end{align}
  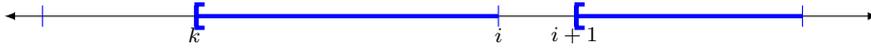
\begin{figure}[h!]
\centering
 \begin{tikzpicture}
\draw[latex-latex] (-7.5,0) -- (4,0) ; 
\foreach \x in  {-7,-5,-1,0,3 } 
\draw[shift={(\x,0)},color=blue] (0pt,4pt) -- (0pt,-4pt);
\draw[ [-, ultra thick, blue] (-5,0) -- (-1,0);
\node at (-5,-0.25) {$k$};
\node at (-1,-0.25) {$i$};
\node at (0,-0.25) {$i+1$};
\draw[ [-, ultra thick, blue] (0,0) -- (3,0);
\end{tikzpicture}
\caption{Terms involving $T_{ki}T_{i+1,b}$ vanish as $x_i = 0$.}
\label{fig:case1}
\end{figure}

 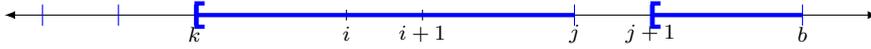
\begin{figure}[h!]
\centering
 \begin{tikzpicture}
\draw[latex-latex] (-7.5,0) -- (4,0) ; 
\foreach \x in  {-7,-6,-5,0,1,3 } 
\draw[shift={(\x,0)},color=blue] (0pt,4pt) -- (0pt,-4pt);
\foreach \x in {-3,-2}
\draw[shift={(\x,0)},color=black] (0pt,2pt) -- (0pt,-2pt);
\draw[ [-, ultra thick, blue] (-5,0) -- (0,0);
\node at (-3,-0.25) {$i$};
\node at (-2,-0.25) {$i+1$};
\node at (-5,-0.25) {$k$};
\node at (-0,-0.25) {$j$};
\node at (1,-0.25) {$j+1$};
\node at (3,-0.25) {$b$};
\draw[ [-, ultra thick, blue] (1,0) -- (3,0);
\end{tikzpicture}
\caption{Terms involving $T_{kj}T_{i+1,b}$ vanish as only one of $T_{kj}$, $T_{i+1,b}$ can be 1.}
\label{fig:case2}
\end{figure}

   \begin{remark}
     The representation in \eqref{eq:cross-terms-representation} for
     the cross terms $x_ix_{i+1}$ can be easily understood via the
     interval partition representation $I_{\bold{x}}$ for
     $\bold{x} \in \mathcal{X}_{app}$ described in Lemma
     \ref{lem:extr-pts-X-app-sched}. For any point
     $\bold{x} \in \mathcal{X}_{app},$ identify the only interval in
     the partition $I_{\bold{x}}$ containing $i$ to be $[k,j].$ Then
     we have that $x_j = 0, x_{j-1} = 1, \ldots, x_k = j-k$ and
     $T_{kj}=1$ (see \eqref{defn-Tkj}). If $i + 1 \in [k,j],$ then
     $x_i = (j-i)$ and $x_{i+1} = j- (i+1)$ and the product
     $x_i x_{i+1} = T_{kj} (j-i) (j-(i+1))$. On the other hand, if
     $i+1$ does not belong to the interval $[k,j],$ we have $x_i = 0;$
     consequently, again $x_i x_{i+1} = T_{kj} (j-i) (j-(i+1)).$ Since
     only one element, $T_{kj},$ in the collection
     $\{T_{ab}: a \leq i \leq b\}$ equals one, the representation in
     \eqref{eq:cross-terms-representation} holds. While the
     representation for the square terms in \eqref{eq:xi-sq} has been
     known in the literature (see, for example, \cite{Mak2015}), the
     representation for the specific cross terms in
     \eqref{eq:cross-terms-representation} has been explicitly characterized, as far as we
     know, for the first time in this paper.
  \label{rem-cross-terms-intuitive}
\end{remark}
\noindent
Combining the observation in \eqref{eq:cross-terms-representation}
with that in \eqref{eq:xi-sq} we obtain
   \begin{align*}
     C_{app} &= conv\left\{ \left(x_1,   \ldots , x_n , x_1^2, \ldots , x_n^2,
     x_1x_2, x_3x_4, \ldots ,x_{n-1}x_{n}\right): \bold{x} \in
     \mathcal{X}_{app}\right\} \\
     &= conv \bigg\{\left(p_1,\ldots,p_n,X_{11},\ldots,X_{nn},X_{12},X_{34},\ldots,X_{n-1,n}\right)
              \in \mathbb{R}^{5n/2} : \nonumber\\
            &\quad \quad\quad\quad  p_i =   \sum_{k=1}^i\sum_{j=i}^{n+1} T_{kj}(j-i),  \ X_{ii} =
              \sum_{k=1}^i\sum_{j=i}^{n+1} T_{kj}(j-i)^2,\ \text{ for } i \in
              [n], \nonumber\\
            &\quad \quad \quad\quad  X_{i,i+1} = \sum\limits_{k=1}^i \sum\limits_{j={i+1}}^{n+1}
              T_{kj} (j-i)(j-(i+1)), \ \text{ for } i \in [n], i \text{ odd}, \nonumber\\
            &\quad \quad \quad\quad  \sum_{k=1}^i\sum_{j=i}^{n+1}T_{kj} = 1, \text{ for } i \in
              [n], \quad
              T_{kj} \in \{0,1\} \text{ for } 1 \leq k \leq j   \leq n+1 \bigg\},
   \end{align*}
   as a consequence of Lemma \ref{lem:extr-pts-X-app-sched}.  Further, total unimodularity of the constraints  over $T$
   verifies the representation for $C_{app}$ in the statement of
   Proposition \ref{prop:capp-char}. \hfill$\Box$
\end{proof}
With this characterization of the set $C_{app}$ in Proposition
\ref{prop:capp-char}, observe that the statement of Theorem
\ref{thm:SmallSDP-app-sched} follows as a consequence of the
formulation in \eqref{eq:inter-small-sdp-app-sched}. This completes
the proof of Theorem \ref{thm:SmallSDP-app-sched}. \hfill$\Box$

\begin{remark}
\emph{Mean-variance bound:}  For any given schedule $\bold{s} \in S,$ the worst-case expected
  total waiting time,
  \begin{align}
    \label{eq:app-sched-mom-prob-mv}
    Z_{mv}^\ast(\bold{s}) =
    \sup\big\{\mathbb{E}_\theta\left[f(\tilde{\bold{u}},\bold{s})\right]
    : \  &\mathbb{E}_\theta\left[\tilde{u}_i
      \right] =\mu_i, \mathbb{E}_\theta[\tilde{u}_i^2] =\Pi_{ii}, \text{ for }
      i \in [n]\big\},
  \end{align}
  that is consistent with given mean and variance information of the
  service times $\{\tilde{u}_i: i \in [n]\}$, can be computed by
  solving the semidefinite programming formulation below in
  \eqref{opt:SmallSDP-app-sched-mean-var}. This formulation results
  from a similar application of Theorem \ref{thm:main} to the simpler
  case where $\mathcal{N}_r = \{r\},$ for $r = 1,\ldots,n.$
\begin{equation}
\label{opt:SmallSDP-app-sched-mean-var}
 \begin{array}{rrlllll}
   Z^\ast_{mv}(\bold{s}) = &\displaystyle \max_{p_i,X_{ii}, Y_{ii}, t_{kj}} &
   \displaystyle \sum_{i=1}^n(Y_{ii} - s_ip_i) \\
   &\mbox{s.t} &
\begin{bmatrix}
 1 & \mu_i & p_i   \\
 \mu_i  & \Pi_{ii} &   Y_{ii}  \\
  p_i & Y_{ii}&   X_{ii} \\
\end{bmatrix} \succeq 0,& \text{ for } i \in [n],\\
& & \displaystyle p_i = \sum_{k=1}^i \sum_{j=i}^{n+1}
t_{kj} (j-i), & \text{ for } i \in [n],\\
& & \displaystyle X_{ii} =  \sum_{k=1}^i \sum_{j=i}^{n+1}
t_{kj} (j-i)^2, & \text{ for } i \in [n],\\
& & \displaystyle \sum_{k=1}^i \sum_{j=i}^{n+1} t_{kj} = 1, & \text{ for } i \in [n],\\
& & \displaystyle  t_{kj} \geq 0, & \text{ for } 1 \leq k  \leq j
\leq n+1.
\end{array}
\end{equation}
Similar to the formulation in Corollary \ref{cor-schedules-dual}, a
distributionally robust schedule that minimizes the worst-case total
expected waiting time can be found by solving the following
semidefinite program:
\begin{equation}
\label{opt:optimal_schedule_mean_variance}
 \begin{array}{rrllllll}
&  \displaystyle \min_{\bold{s}, \boldsymbol{\eta},
  \boldsymbol{\beta}, \bold{\Gamma},\boldsymbol{\rho},
  \boldsymbol{\delta},  \boldsymbol{\gamma}} & \displaystyle \sum_{i =1}^n \eta_i
+\sum_{i=1}^n\beta_i\mu_i+
 \sum_{i=1}^n \Gamma_{ii}\Pi_{ii} + \sum_{i=1}^{n+1} \rho_i \\
 & \mbox{s.t} &
 \begin{bmatrix}
 2\eta_i & {\beta_i} & {\delta_i + s_i} \\
  {\beta_i} &  2\Gamma_{ii} &   -{1}  \\
    {\delta_i + s_i} & -{1}  & 2\gamma_i \\
  \end{bmatrix} \succeq 0, &\text{ for } i \in [n], \\
  & & \displaystyle \sum_{i=k}^j \rho_i \geq \hspace{-5pt}\sum_{i=k}^{\min\{j,n\}}\hspace{-10pt}\delta_i
  (j-i) + \hspace{-5pt}\sum_{i=k}^{\min\{j,n\}}\hspace{-10pt}\gamma_i (j-i)^2, &
  \text{ for } 1 \leq k \leq j \leq n+1,\\
  & & \displaystyle  \sum_{i=1}^ns_i \leq T, \\
  & & \displaystyle s_i \geq 0, & \text{ for } i \in
  [n].
 \end{array}
 \end{equation}
 The semidefinite program in
 \eqref{opt:optimal_schedule_mean_variance} can be seen as an
 alternative to the second order conic programming formulation in
 \cite{Mak2015} where the problem of appointment scheduling in the
 presence of mean and variance information was considered.
\label{rem:mean-variance-case}
The equivalent SOCP formulation  \cite{Mak2015} is provided next where $\boldsymbol{\sigma^2}$ denotes the vector of variances:
 \begin{equation}
\label{opt:jiawei-app-sched-formulation}
 \begin{aligned}
{Z}_{\text{app}}(\boldsymbol{\mu}, \boldsymbol{\sigma^2}) &= \underset{\bold{s},\beta>0, \alpha, \lambda}{\text{min}}
& & \sum_{i=1}^n \lambda_i + \mu_i \alpha_i + (\mu_i^2 + \sigma_i^2)\beta_i \\
& \text{s.t} & &
\sum_{i=k}^{\min\{n,j\}} \lambda_i \geq  \sum_{i=k}^{\min\{n,j\}} \left( \frac{(\pi_{ij} - \alpha_i)^2} {4\beta_i} - s_i \pi_{ij} \right) \\
& & & \text{         for }1 \leq k \leq n, k \leq j \leq n+1
\end{aligned}
\end{equation}
where $
\pi_{ij}  = j-i , 1 \leq i \leq j \leq n+1.
$

\end{remark}

\begin{remark}
  A representation for cross terms $x_ix_{i+2},$ similar to that in
  \eqref{eq:cross-terms-representation} in terms of variables
  $T_{kj} \in \{0,1\},$ does not result in linear representation in the variables $T_{kj}.$
  To see this, recall the interval partition representation described
  in Lemma \ref{lem:extr-pts-X-app-sched}.  Consider any
  $\bold{x} \in \mathcal{X}_{app}$ such that there exist $k,j$
  satisfying $[k,i+1] \in I_{\bold{x}}$ and
  $[i+2,j] \in I_{\bold{x}}.$ Then $x_i = 1$ and
  $x_{i+2} = j - (i+2),$ in which case
  $x_ix_{i+2} = \sum_{k =1}^{i+1}\sum_{j=i+2}^{n+1}T_{k,i+1}
  T_{i+2,j}(j-(i+2)),$ which cannot be reduced in a straightforward manner to a linear
  representation as in \eqref{eq:cross-terms-representation}.
  \label{rem-cross-geq-2}
\end{remark}



\subsection{Longest path in directed acyclic graphs}
\label{sec:path}
In this section, we examine the problem of computing the expected
length of the longest path between a fixed start node and a sink node
in a directed acyclic graph whose arc lengths are uncertain. A key
application of this longest path problem is to estimate project
completion times using Project Evaluation and Review Technique (PERT)
networks in project management (see, for example,
\cite{VanSlyke1963}).  A PERT network is a directed acyclic graph
representation of a project that consists of several activities with
partially specified precedence relationship among the activities. Our
objective is to tackle the case where the activity durations (arc
lengths) are random, dependent and their joint distribution is not
fully known.

Let ${V} = \{0,\ldots,m-1\}$ denote the set of nodes of a directed
acyclic graph $G.$ Suppose that the nodes $0$ and $m-1$ represent the
start and sink nodes. Let $\mathcal{A}$ denote the set of arcs in $G$
and $c_{ij}$ denote the length of arc $(i,j)$ between nodes $i$ and
$j.$ If $G$ is a PERT network, the nodes $0$ and $m-1$ represent the
start and end of the project; the length of the longest path between
nodes $0$ and $m-1$ represents the project completion duration. The
length of the longest path can be represented as the optimal objective
value of the following combinatorial optimization problem:
\begin{equation}
\label{opt:det_pert}
\begin{array}{rrllllll}
  Z(\bold{c}) = & \max  & \displaystyle \sum_{(i,j) \in \mathcal{A}} c_{ij} x_{ij}  \\
  & \mbox{s.t} & \displaystyle \sum_{j: (i,j) \in \mathcal{A}} x_{ij} - \sum_{j:
    (j,i) \in \mathcal{A}} x_{ji} =
\begin{cases}
\ \ 1, \quad &\mbox{if } i=0, \\
 -1,  &\mbox{if } i=m-1, \\
\ \ 0, & \mbox{otherwise, } \\
\end{cases} \\
& &  x_{ij} \in \{0,1\},  \quad\text{ for }  (i,j) \in
\mathcal{A}.
\end{array}
\end{equation}
If the arc lengths $(c_{ij})_{(i,j) \in \mathcal{A}}$ are known,
$Z(\bold{c})$ can be computed in polynomial-time by solving the linear
programming relaxation of the formulation in \eqref{opt:det_pert} due
to the total unimodularity of the underlying constraint matrix.

On the other hand, if the arc lengths are random, exact computation of
the expected length of the longest path is known to be \#P-hard even with the
assumption of independence among arc lengths (see
\cite{Hagstrom1988}). For specialized graph structures such as
series-parallel graphs, it has been shown in
\cite{Ball1995,Mohring2001} that the expected length of the longest
path can be computed in time polynomial in the size of the graph and
the number of points in the discrete support of the arc lengths.

In the absence of the knowledge of the entire joint distribution of
the arc lengths, the distributionally robust formulations in
\cite{Bertsimas2004,Bertsimas2006} result in polynomial-time solvable
bounds for the project duration when the marginal moments of arc
lengths are specified. A natural approach to specify correlation
information in PERT networks, in order to obtain tighter bounds, is to
consider all the activities that enter a node to be related and
therefore specify correlation information among all activities that
enter a node. Indeed, such a partition formed by sets of incoming arcs
into nodes have been considered for specifying marginal distribution
information in
\cite{doi:10.1287/opre.10.6.808,doi:10.1002/net.3230070407,Doan2012}.
Theorem \ref{thm:longest-path} below identifies a polynomial-time
solvable formulation for evaluating the maximum possible (worst-case)
expected project duration in the presence of mean and covariance
information of activity durations whose arcs enter the same node.

To fix notation, let $n$ be the cardinality of the set $\mathcal{A}$
of arcs and $R = m-1$. For the given directed acyclic graph $G,$
consider the following partition of $\mathcal{A},$
\begin{align}
  \mathcal{N}_r = \{i:(i,r) \in \mathcal{A}\}, \quad\quad\text{for }
  r = 1,\ldots,R,
\label{eq:partition-PERT}
\end{align}
formed by considering sets of arcs that enter node $r,$ for
$r = 1,\ldots,m-1.$ Let
$\tilde{\bold{c}} = (\tilde{c}_{ij})_{(i,j) \in \mathcal{A}}$ be the
random vector of arc lengths and
$\tilde{\bold{c}}^r = (\tilde{c}_{ir})_{i:\,(i,r) \in \mathcal{A}}$ be
the random subvector of arc lengths of arcs entering node $r,$ for
$r = 1,\ldots,R.$ Given that the expected value of
$\tilde{\bold{c}}$ is $\boldsymbol{\mu}$ and that of
$\tilde{\bold{c}}^r(\tilde{\bold{c}}^r)'$ is $\bold{\Pi}^r$ for every
$r \in \{1,\ldots,m-1\},$ our objective is to evaluate,
\begin{align}
  Z^\ast_{path} = \sup\left\{
  \mathbb{E}_\theta\left[Z(\tilde{\bold{c}})\right]:
  \mathbb{E}_\theta[\tilde{\bold{c}}] = \boldsymbol{\mu},\,
  \mathbb{E}_\theta[\tilde{\bold{c}}^r(\tilde{\bold{c}}^r)'] =
  \bold{\Pi}^r\text{ for } r \in [R\,], \, \theta \in
  \mathcal{P}(\mathbb{R}^{n}) \right\},
  \label{eq:defn-zpath}
\end{align}
where $Z(\cdot)$ is specified as in \eqref{opt:det_pert}.
\begin{theorem}
  \label{thm:longest-path}
   $Z^\ast_{path}$ can be evaluated as the
  optimal objective value of the following semidefinite program:
   \begin{equation}
\label{opt:SmallSDP-path}
 \begin{array}{rrllllll}
{Z}^\ast_{path} &= \displaystyle \max_{\substack{\bold{p}^r, \bold{X}^r, \bold{Y}^r} } & \displaystyle \sum_{r=1}^{m-1} trace(\bold{Y}^r) \\
&\mbox{s.t} &
\begin{bmatrix}
 1 & \boldsymbol{\mu^r}' & {\bold{p}^r}'  \\
\boldsymbol{\mu}^r & \bold{\Pi}^r & {\bold{Y}^r}'\\
   \bold{p}^r  & \bold{Y}^r& \bold{X}^r\\
 \end{bmatrix} \succeq 0, \;\quad\quad \text{ for } r \in
 \{1,\ldots,m-1\},
 \\
 & & \displaystyle \sum_{j: (i,j) \in \mathcal{A}} p_{ij} - \sum_{j:
   (j,i) \in \mathcal{A}} p_{ji} =
\begin{cases}
\ \ 1, &\mbox{if } i=0, \\
-1, & \mbox{if } i=m-1, \\
\ \ 0, & \mbox{otherwise,} \\
\end{cases} \\
& & \displaystyle X^r_{jk} =
 \begin{dcases}
  p_{ij}, & \text{ if } j=k, \\
  0, & \text{otherwise,}\\
  \end{dcases}\quad \text{for } r=1, \ldots, m-1,\\
 & & \displaystyle p_{ij} \geq 0, \quad \text{ for } (i,j) \in
 \mathcal{A}.
\end{array}
\end{equation}
\end{theorem}
\begin{proof}
  Let us use $\mathcal{X}_{path}$ to denote the feasible region to the
  formulation \eqref{opt:det_pert}. Then as an
  application of Theorem \ref{thm:main}, $Z^\ast_{path}$ can be
  written as the optimal objective value of the semidefinite program
  in \eqref{opt:SmallSDP}. To efficiently represent the convex hull
  constraint in \eqref{opt:SmallSDP}, observe that for any
  $\bold{x} \in \mathcal{X}_{path},$
  \begin{align}
    x_{ir}x_{jr} =
    \begin{cases}
      x_{ir}, \quad&\text{ if } i = j,\\
      0, &\text{ if } i \neq j,
    \end{cases}
      \label{eq:cross-terms-path}
  \end{align}
  for every $i,j$ such that $i,j \in \mathcal{N}_r.$ This follows from
  the observation that any path from $0$ to $m$ that passes through
  $r$ can contain only one of the arcs
  $\{(k,r): (k,r) \in \mathcal{A}\}.$ To see this explicitly from the
  constraints in \eqref{opt:det_pert}, observe that if $\bold{x}$ is
  such that $\sum_{k:(k,r) \in \mathcal{A}}x_{kr} = 1,$ then as
  $x_{kr} \in \{0,1\},$ only one of $\{x_{kr}: (k,r) \in \mathcal{A}\}$
  equals 1. Therefore $x_{ir}x_{jr} = 0$ if $i \neq j.$ On the other
  hand, $x_{ir}^2 = x_{ir}$ as $x_{ir} \in \{0,1\},$ thus verifying
  \eqref{eq:cross-terms-path}. As a result of this and total unimodularity of
  the constraints in formulation \eqref{opt:det_pert},
  \begin{align*}
    conv&\left\{\left(\bold{x}, \bold{x}^1{\bold{x}^1}', \ldots,
          \bold{x}^{m-1}{\bold{x}^{m-1}}'\right): \bold{x} \in
          \mathcal{X}_{path}\right\}\\
        &\quad= conv\left\{\left(\bold{x}, \text{Diag}(\bold{x}^1),\ldots,
          \text{Diag}(\bold{x}^{m-1})\right): \bold{x} \in
          \mathcal{X}_{path}\right\},\\
        &\quad= \left\{\left(\bold{p}, \text{Diag}(\bold{p}^1), \ldots,
          \text{Diag}(\bold{p}^{m-1})\right): \bold{p} \in
          conv(\mathcal{X}_{path})\right\},
  \end{align*}
  where Diag$(\bold{x}^r)$ denotes the $n_r \times n_r$ diagonal
  matrix formed with
  elements from the subvector $\bold{x}^r.$ Since the convex hull of
  $\mathcal{X}_{path}$ is simply the collection of points $\bold{p} =
  (p_{ij})_{(i,j) \in \mathcal{A}}$ such that $p_{ij} \geq
  0$ and
  \begin{align*}
    \sum_{j: (i,j) \in \mathcal{A}} p_{ij} - \sum_{j:
    (j,i) \in \mathcal{A}} p_{ji} =
    \begin{cases}
      \ \ 1, \quad &\mbox{if } i=0, \\
      -1,  &\mbox{if } i=m-1, \\
      \ \ 0, & \mbox{otherwise, } \\
    \end{cases},
  \end{align*}
  the constraints in the formulation \eqref{opt:SmallSDP-path} are
  equivalent to those in \eqref{opt:SmallSDP}. This completes the
  proof of Theorem \ref{thm:longest-path}.  \hfill$\Box$
\end{proof}

\subsection{Linear assignment problem}
\label{sec:asg}
In this section, we consider the linear assignment problem (see
\cite{Kuhn1955,Munkres1957}), where $m$ entities belonging to a set
$V$ need to be assigned to $m$ entities in a set $U$.  The entities in
the sets $U$ and $V$ can be thought, respectively, as candidates and
jobs that need to be performed by the candidates. Each candidate must
be assigned to exactly one job and each job must be assigned to
exactly one candidate. In the version we consider, only candidates can
have preferences for jobs and let $c_{ij}$ be the preference of
candidate $i$ for job $j.$

The described information can be represented as a weighted undirected
bipartite graph with weights given by
$\bold{c} = (c_{ij})_{(i,j): i \in U, j \in V}.$ Among the $m!$
possible assignments of entities in $V$ to $U,$ the objective is to
identify an assignment that maximizes the sum of preferences of all
candidates. Hereafter, we address the sum of preferences of an
assignment as the ``total welfare'' of the assignment.  An assignment
that maximizes the total welfare can be obtained by solving the
following combinatorial optimization problem:
\begin{equation}
\label{opt:det_assignment}
\begin{array}{rrllllll}
Z(\bold{c})= & \displaystyle \max_{x_{ij}} & \displaystyle \sum_{i \in U} \sum_{j \in V} c_{ij} x_{ij} & \\
& \mbox{s.t} & \displaystyle \sum_{j \in V} x_{ij} =1, & \text{ for } i \in U,\\
& & \displaystyle \sum_{i\in  U} x_{ij} =1, & \text{ for } j \in V,\\
& & \displaystyle x_{ij} \in \{0,1\},  & \text{ for } i \in U, j \in V.
\end{array}
\end{equation}

Similar to the formulation for computing the length of the longest
path in \eqref{opt:det_pert}, the constraints in the formulation
\eqref{opt:det_assignment} are totally unimodular. Therefore a linear
programming relaxation can be used to identify an optimal assignment
in polynomial-time.

When the preferences are random, there have been attempts in
\cite{Pardalos1993,AlmSorkin2002} to develop an understanding of the
expected total welfare (sum of preferences) of the optimal
assignments. Since the deterministic formulation in
\eqref{opt:det_assignment} is solvable in polynomial-time, the
distributional robust bound with marginal moments can be solved in
polynomial-time as well to result in tight bounds (see
\cite{Bertsimas2004}). In this section, we employ Theorem
\ref{thm:main} to compute tight bounds that are applicable when partial
correlation information is known in addition to the marginal moments.

Identifying the entries in $U$ with $\{1,\ldots,m\},$ we take
$\mathcal{N}_r = \{(r,j): j \in V\},$ for $r \in [m].$ This
corresponds to the setting where the correlation of preferences
between any two jobs for the same candidate is known, but correlation
across candidates is not known. Let
$\tilde{\bold{c}} = (\tilde{c}_{ij})_{i: i \in U, j \in V}$ be the
random vector of preferences and
$\tilde{\bold{c}}^i = (\tilde{c}_{ij})_{j \in V}$ be the random
subvector of $\tilde{\bold{c}}$ when the indices are restricted to the
subset $\mathcal{N}_i.$ We aim to evaluate the bound,
\begin{align}
  Z^\ast_{lap} = \sup\left\{
  \mathbb{E}_\theta\left[Z(\tilde{\bold{c}})\right]:
  \mathbb{E}_\theta[\tilde{\bold{c}}] = \boldsymbol{\mu},\,
  \mathbb{E}_\theta[\tilde{\bold{c}}^r(\tilde{\bold{c}}^r)'] =
  \bold{\Pi}^r\text{ for } r \in [R\,], \, \theta \in
  \mathcal{P}(\mathbb{R}^{n}) \right\},
  \label{eq:defn-zpath}
\end{align}
where $Z(\cdot)$ is given by \eqref{opt:det_assignment}.  As an
alternative to the partition considered, one could consider the
partition where we identify the entities in $V$ with $\{1,\ldots,m\}$
and take $\mathcal{N}_r = \{(i,r): i \in U\},$ for $r \in [m].$ In
settings where $-c_{ij}$ can be interpreted as the cost for assigning
job $j$ to candidate $i,$ this partition corresponds to knowing
correlation between costs for the same job when performed by different
candidates. The following observation can be replicated for this
partition as well.

\begin{theorem}
  \label{thm:longest-asg}
  Suppose that $Z^\ast_{lap}$ is defined as in
  \eqref{eq:defn-zpath}. Then $Z^\ast_{lap}$ can be evaluated as the
  optimal objective value of the following semidefinite program:
   \begin{equation}
\label{opt:SmallSDP-asg}
 \begin{array}{rrlllll}
\displaystyle  {Z}^\ast_{lap} &= \displaystyle \max_{\substack{\bold{p}^r, \bold{X}^r, \bold{Y}^r}} & \displaystyle \sum_{r=1}^{m-1} trace(\bold{Y}^r) \\
&\mbox{s.t} &
\begin{bmatrix}
 1 & {\boldsymbol{\mu}^r}' & {\bold{p}^r}'  \\
\boldsymbol{\mu}^r & \bold{\Pi}^r & {\bold{Y}^r}'\\
   \bold{p}^r  & \bold{Y}^r& \bold{X}^r\\
 \end{bmatrix} \succeq 0, &\text{ for } r \in [m],
 \\
 & & \displaystyle \sum_{j \in V} p_{ij}  = 1, & \text{ for } i \in U,\\
 & & \displaystyle \sum_{i \in U} p_{ij}  = 1, & \text{ for } j \in V,\\
 && \displaystyle  X^r_{jk} =
 \begin{dcases}
  p_{rj}, & \text{ if } j=k, \\
  0, & \text{otherwise,}\\
  \end{dcases}\quad \text{ for } r=1, \ldots,m, \\
 && \displaystyle p_{ij} \geq 0, &
    \text{ for } i \in U, j \in V.
 \end{array}
\end{equation}
\end{theorem}
\begin{proof}
  As in the proof of Theorem \ref{thm:longest-path}, let
  $\mathcal{X}_{lap}$ be the bounded feasible region to the
  formulation \eqref{opt:SmallSDP-asg}. Then as an application of
  Theorem \ref{thm:main}, $Z^\ast_{lap}$ can be written as the
  optimal objective value of the semidefinite program in
  \eqref{opt:SmallSDP}. To efficiently represent the convex hull
  constraint in \eqref{opt:SmallSDP}, observe that for any
  $\bold{x} \in \mathcal{X}_{lap},$
  \begin{align*}
    x_{ij}x_{ik} =
    \begin{cases}
      x_{ij}, \quad&\text{ if } j = k,\\
      0, &\text{ if } j \neq k,
    \end{cases}
  \end{align*}
  for every $i \in U,$ $j, k \in V.$ This is because, as in the proof
  of Theorem \ref{thm:longest-path}, only one of $\{x_{ij}: j \in V\}$
  equals 1, for every $i \in U;$ here recall the constraint,
  $\sum_{j \in V}x_{ij} = 1,$ that dictates that only one entity from
  $V$ is assigned exactly to every $i \in U.$  With
  $\bold{x}^i = (x_{ij})_{j \in V},$ we
  obtain, 
  \begin{align*}
    conv&\left\{\left(\bold{x}, \bold{x}^1{\bold{x}^1}', \ldots,
    \bold{x}^{m}{\bold{x}^{m}}'\right): \bold{x} \in
    \mathcal{X}_{lap}\right\}\\
    &\quad\quad= conv\left\{ \left(\bold{x}, \text{Diag}(\bold{x}^1),\ldots,
    \text{Diag}(\bold{x}^{m})\right): \bold{x} \in
    \mathcal{X}_{lap}\right\},\\
    &\quad\quad= \left\{\left( \bold{p}, \text{Diag}(\bold{p}^1), \ldots,
      \text{Diag}(\bold{p}^{m})\right): \bold{p} \in
      conv(\mathcal{X}_{lap})\right\},
  \end{align*}
  where Diag$(\bold{x}^i)$ denotes the $m \times m$ diagonal matrix
  formed with elements from the subvector $\bold{x}^i.$ Since, due to total unimodularity, the
  convex hull of $\mathcal{X}_{lap}$ is simply the collection of
  points $\bold{p} = (p_{ij})_{i \in U, j \in V}$ such that
  $p_{ij} \geq 0,$ $\sum_{j \in V}p_{ij}=1$ and
  $\sum_{i \in U}p_{ij}=1,$ the constraints in the formulation
  \eqref{opt:SmallSDP-asg} are equivalent to those in
  \eqref{opt:SmallSDP}. This completes the proof of Theorem
  \ref{thm:longest-asg}.  \hfill$\Box$
\end{proof}

\section{Numerical results}
\label{sec:computations}
In this section, we report the results of numerical experiments for
the appointment scheduling formulation considered in Section
\ref{sec:app-sched}. We compare the performance of
the semidefinite programming formulation in Theorem \ref{thm:SmallSDP-app-sched} (which we refer to as \textit{Non-overlapping}), with the
following three alternatives:
\begin{itemize}
\item[a)] The mean-variance formulation is solved using the SOCP reformulation \eqref{opt:jiawei-app-sched-formulation} originally proposed
  in \cite{Mak2015}. This approach, addressed
  ``\textit{Mean-Variance}'' in the discussions that follow, provides an upper bound for $Z^\ast_{app}(\bold{s})$ in Theorem \ref{thm:SmallSDP-app-sched}.
\item[b)] For the second alternative, we solve for $Z^\ast_{app}(\bold{s})$
  by maximizing over the unspecified covariance entries of
  $\tilde{\bold{u}}$ in the formulation \eqref{opt:cp-app-sched-formulation} (originally proposed in \cite{Kong2013}) that
  assumes the knowledge of the mean and entire covariance matrix of
  $\tilde{\bold{u}}$. The exact formulation in \cite{Kong2013}
  involves a completely positive constraint, which is then relaxed to
  a doubly nonnegative matrix constraint for tractability (see
  \cite{Kong2013}), thus resulting only in an upper bound for
  $Z_{app}^\ast(\bold{s}).$ In \Cref{simulations:bound} where we restrict our attention to the distributionally robust bound, we use this formulation, However in \Cref{simulations:schedule}, we use the dual formulation  \eqref{opt:cp-app-sched-formulation-dual}, where instead of the copositivity requirement $\bold{M} \in  {\cal C^*}(\mathbb{R}_+^{2n+1})$, we use $ \bold{M} = \bold{S} + \bold{N}, \bold{S} \succeq 0 , \bold{N} \text{ non-negative}$ as an approximation. We address the primal as well as dual formulations as
  ``\textit{DNN-relaxation}'' in the discussion below.
\item[c)] The exact value of $Z_{app}^\ast(\bold{s})$ is also computed for the formulation
  \eqref{opt:karthik-reduced-sdp-formulation} where the extreme points
  corresponding to \eqref{opt:det-large} are explicitly enumerated and only the partial moment information is assumed to be known. The explicit enumeration of the extreme points involves introduction of new scalar variables $\alpha_\bold{x}$ for each extreme point $\bold{x}$ such that $\sum_{\bold{x}} \alpha_\bold{x} = 1$ and $\alpha_{\bold{x}} \geq 0$.  The convex hull constraint in formulation \eqref{opt:karthik-reduced-sdp-formulation} is captured using the following constraints:
$
  \bold{p} = \sum\limits_{\bold{x} \in \mathcal{X}_{app}} \alpha_{\bold{x}}\bold{x} ,
   \bold{X} = \sum\limits_{\bold{x} \in \mathcal{X}_{app}} \alpha_{\bold{x}}\bold{x}\bold{x}',
   \sum_{\bold{x} \in \mathcal{X}_{app}} \alpha_\bold{x} = 1 ,
   \alpha_{\bold{x}} \geq 0 \forall \bold{x} \in \mathcal{X}_{app}
$.
  Since the number of extreme points grows exponentially with $n,$
  this approach is feasible only for small values of $n.$ This exact
  approach, labeled as ``\textit{Large-SDP}'', is feasible in our
  computational setup only for $n \leq 9.$
\end{itemize}

We also tested a recently proposed alternate approximation scheme proposed in
\cite{BomzeCDL17} in place of the doubly nonnegative matrix based
relaxation for approximating the completely positive constraint in the
exact formulation in \cite{Kong2013}. The results obtained were
identical to the approach labeled above as ``\textit{DNN-relaxation}''
and hence we only report the results of \textit{DNN-relaxation}. All
experiments were run on MATLAB using SDPT3
solver\footnote{http://www.math.nus.edu.sg/~mattohkc/sdpt3.html}
\cite{Toh1999,Tutuncu2003} and YALMIP
interface\footnote{https://yalmip.github.io/}. For the SOCP based
approach in \cite{Mak2015} that is labeled as
``\textit{Mean-Variance}'', the solver used is SeDuMi \cite{SeDuMi}.

\subsection{Comparison of worst-case expected total waiting times}
\label{simulations:bound}
Assuming that the correlation coefficient between service times
$\tilde{u}_i$ and $\tilde{u}_{i+1}$ equals $\rho,$ for every $i$ in
$\{1,3,\ldots, n-1\}$, we compare the objective value of the
formulation in Theorem \ref{thm:SmallSDP-app-sched} with that of the
alternative approaches described above, for various values of $\rho$
in the interval $[-1,1].$ We report objective values averaged over 50
independent runs, where in each run, the means and variances of
$\tilde{{c}}_i(\bold{s}) = \tilde{{u}}_i - {s}_i$ are taken to be
independent realizations of random variables uniformly distributed in
the intervals $[-2,2]$ and $(0,5]$ respectively. For all the results in the current subsection, since we are interested in only the bound computation, we set $\bold{s}=0$ in all the  formulations.

See Figure \ref{fig:app-sched-bounds} for a comparison of the
ratio of average objective values of our formulation in Theorem
\ref{thm:SmallSDP-app-sched} and the \textit{Large-SDP} approach for
$n = 6.$ \Cref{tab:app-sched-bound-rho} gives the min, max and mean ratios for various formulations. Since our formulation is exact, it is not surprising that
the ratio is 1 for all values of $\rho.$ The ratios resulting by
comparing average objective values of \textit{Mean-Variance} and
\textit{DNN-relaxation} approaches with the exact \textit{Large-SDP}
approach are also reported in Figure \ref{fig:app-sched-bounds}. The
variability in the ratio for the \textit{Mean-Variance} approach can
be inferred from the error bars in \Cref{fig:app-sched-bounds}. The
growth in gap between the objective values of the
\textit{Mean-Variance} approach and our partial covariance based
approach in \eqref{eq:app-sched-mom-prob}, as $n$ increases can be
inferred from Figure \ref{fig:app-sched-bounds-n}.

 \begin{table}[h!]
 \caption{Bound ratios over Large-SDP bound for various approaches to DR appointment scheduling for various $\rho$ values, n=6. 50 runs were performed with random means in [-2,2] and variances in (0,5].
 }
 \label{tab:app-sched-bound-rho}
 \centering
 \begin{tabular}{|c|c|c|c|c|c|c|c|c|c|c|c|c|}
 \hline

	&  \multicolumn{3}{|c|}{Mean-variance} &   \multicolumn{3}{|c|}{Our Approach}&   \multicolumn{3}{|c|}{DNN Relaxation}\\\hline
$\rho$	& mean &	 min	& max & 	mean 	&min &	max &	mean  	& min &	max \\\hline
-1.0	 &1.489&		1.054&	2.028&	1&1	&1	&1.001	&1	&1.008\\\hline
-0.7	& 1.251	&1.036	&1.492&	1&	1	&1&	1.001		&1	&1.006 \\\hline
-0.3	& 1.141	&1.023	&1.285&	1		&1	&1	&1.001		&1	&1.004 \\\hline
0.0	 &1.088	&1.016	&1.185	&1	&1&	1	&1.001&	1.001	&1.007 \\\hline
0.3	&1.051	&1.010&	1.111&	1		&1	&1&	1.001	&	1	&1.002 \\\hline
0.7	&1.017	&1.001&	1.039&	1		&1&	1&	1.001	&1	&1.001 \\\hline
1.0&	1.010&	1&	1.055&	1		&1&	1	&1.002&1	&1.056 \\\hline
 \end{tabular}
 \end{table}

\begin{figure}[h!]
\begin{centering}
 \begin{subfigure}{0.43\textwidth}
   \includegraphics[scale=0.4]{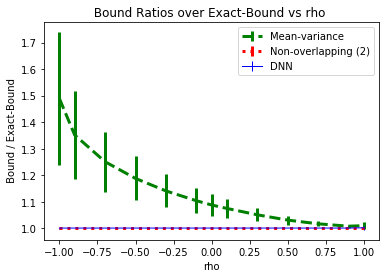}
   \caption{\centering{Ratio of objective value and
       $Z_{app}^\ast(\bold{s})$ for various values of $\rho$, $n=6$}}
      \label{fig:app-sched-bounds}
    \end{subfigure}~~~
      \begin{subfigure}{0.43\textwidth}
    \includegraphics[scale=0.4]{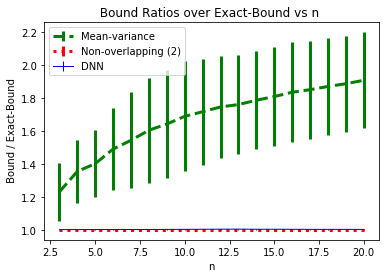}
     \caption{\centering{Ratio of objective value and
       $Z_{app}^\ast(\bold{s})$ for various values of $n$, $\rho=-1$}}
      \label{fig:app-sched-bounds-n}
    \end{subfigure}
      \caption{Bound Ratios of various approaches.}
\end{centering}
 \end{figure}
 It is evident from Figure \ref{fig:app-sched-bounds} that the bound
 resulting from mean-variance formulation in
 \eqref{eq:app-sched-mom-prob} is significantly higher than
 $Z_{app}^\ast(\bold{s})$ for negative values of $\rho.$ As $\rho$
 approaches $1$, the bound resulting from \textit{Mean-Variance}
 approach appears to coincide with $Z^\ast_{app}(\bold{s})$.

 While the numerical results appear to suggest that the
 distributionally robust formulation with partial correlation
 information offers a behaviour similar to that of the
 \textit{Mean-Variance} approach as $\rho \rightarrow 1,$ it is
 worthwhile to note that the correlation coefficients between
 $\tilde{c}_i$ and $\tilde{c}_{i+1}$ need not equal 1 for the
 worst-case distribution that attains the supremum in the
 \textit{Mean-Variance} formulation
 \eqref{eq:app-sched-mom-prob}. Indeed, given marginal distributions,
 for objective functions that are supermodular in its random variable
 arguments, it is well known that the comonotone joint distribution
 maximises the expectation (see, for example,
 \cite{Lorentz1953,cambanis1976inequalities,galichon2016optimal}). However
 this comonotone joint distribution may very well be such that the
 correlation coefficients between its components are lesser than 1.
This also explains the reason why the mean-variance bound need not exactly match the Large-SDP bound for $\rho=1$ (see the last row in \Cref{tab:app-sched-bound-rho}).

 From \Cref{fig:app-sched-bounds,fig:app-sched-bounds-n}, we also observe that the \textit{DNN-relaxation} approach
 consistently gives a good approximation ratio (close to 1, see \Cref{tab:app-sched-bound-rho} for specific values), though it
 tends to be computationally expensive for large values of $n;$ see
 Figure \ref{fig:app-sched-performance-B-2} for comparison of
 execution times for the different approaches considered.

 \begin{minipage}{\textwidth}
 \begin{minipage}{0.43\textwidth}
\includegraphics[scale=0.38]{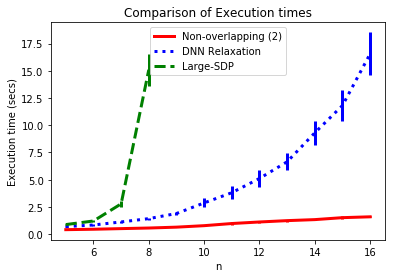}
\captionof{figure}{Execution times in seconds of various approaches
  with $n$
    \label{fig:app-sched-performance-B-2}}
\end{minipage} ~~~
\begin{minipage}{0.45\textwidth}
 \begin{tabular}{|c|c|c|c|}
 \hline
$n$ & Mean	& Min  & Max  \\\hline
30 &	 8.397 	& 8.052 & 	 8.835 \\\hline
40 	& 19.565 	& 18.712 	& 21.127 \\\hline
50 	& 41.215	   &	 38.515 &	 48.330 \\\hline
60 	& 78.533 &	  75.563 	& 82.552 \\\hline
70 	& 129.533 		& 122.533 	& 142.875 \\\hline
80 	& 227.400 	 &	 206.607 &	 244.174 \\\hline
90 	& 416.586  &	 343.712 	& 478.861 \\\hline
100 	&  672.803 	 &	 611.037	& 716.489 \\\hline
\end{tabular}
\captionof{table}{Execution times (in sec) for solving the
  semidefinite program in Theorem \ref{thm:SmallSDP-app-sched}}
\label{tab:exec-times-large-n}
  \end{minipage}
  \end{minipage}

  It can be inferred from Figure \ref{fig:app-sched-performance-B-2} that the
  \textit{Large-SDP} approach is computationally prohibitive for large
  values of $n.$ The mean, minimum and maximum of observed execution
  times of the semidefinite program in Theorem
  \ref{thm:SmallSDP-app-sched} are provided for larger values of $n$
  in \Cref{tab:exec-times-large-n}. Even for $n = 100$, the average
  execution time of our approach is only 672 seconds (roughly $11$
  minutes).

\subsection{Comparison of optimal schedules}
\label{simulations:schedule}
We next compare the optimal schedule obtained using the semidefinite
program in Corollary \ref{opt:optimal_schedule} with those obtained
from the \textit{Mean-Variance} and \textit{DNN-relaxation}
approaches. For this purpose, we consider $n=20$ patients, all with
mean processing duration $2$ and standard deviation $0.5$. We take
$T = 45$ units to be the time within which the schedules need to be
fit. Figure \ref{fig:schedules-mean-blk-rho-1} -
\ref{fig:schedules-mean-blk-rho-0-large-t} portray the schedules,
respectively, for the cases where the correlation coefficient between
$\tilde{u}_i$ and $\tilde{u}_{i+1},$ for $i \in \{1,3,\ldots,n-1\}$ is
given by $\rho = 1,-1, 0$ and $-0.5.$ In order to understand the differences
in the optimal schedules when the full covariance matrix is known, we
plot the schedules given by the \textit{DNN-relaxation} approach for
the specific instance where the covariance entries that are not
specified are set to 0. We use the label \textit{DNN-Full covariance} for this scenario and \textit{DNN-Non-overlapping}  for the DNN-relaxation with partial moments.

Interestingly, for negative values of $\rho,$ we observe that the
inclusion of partial correlation information results in optimal
schedules that are considerably different (in the relative durations
allotted for earlier and later patients) when compared to those
resulting from the \textit{Mean-Variance} approach that assumes only
the knowledge of mean and variance (see Figure
\ref{fig:schedules-mean-blk-rho--1}). For the extreme case where
$\rho = -1,$ we observe that the worst-case waiting time,
$Z^\ast_{app},$ in the presence of partial correlation information is
4.116; this quantity is much smaller when compared to the worst case
expected total waiting time of 25.615 for the \textit{Mean-Variance}
approach where the partial correlation information is not included in
the formulation. Moreover, we observe that employing the optimal
schedule resulting from the mean-variance approach increases the
worst-case waiting time, $Z^\ast_{app}(\cdot),$ by nearly 100\% over
the optimal $Z^\ast_{app}.$ On the other hand, employing the optimal
schedule from our formulation
\eqref{opt:optimal_schedule_mean_variance} results in about 30\%
increase in the worst-case waiting time $Z^\ast_{mv}(\cdot).$ Such
stark changes in the structure and objective value for optimal
schedules are typically not observed for nonnegative values of $\rho$
(see Table \ref{tab:swap-schedules}).

 \begin{table}[h!]
   \caption{Mean percentage increase in the worst-case waiting time
     $Z^\ast_{app}(\cdot)$ when the  optimal schedule from
     \textit{Mean-Variance} approach is used instead of the optimal
     schedule that minimizes $Z^\ast_{app}(\bold{s})$, and vice versa,
     for $n = 20$ and cases $\rho=-1,0$ and $1.$ The rows indicate schedules and columns
     indicate the DRO formulation used: M-V for the objective, $Z^\ast_{mv}(\cdot),$ of the
     \textit{Mean-Variance} approach and P-C for the objective,
     $Z^\ast_{app}(\cdot),$  that  also includes the knowledge of partial
     correlations.}
 \label{tab:swap-schedules}
 \centering
  \begin{tabular}{|c|c|c||c|c||c|c|}
   \hline
   & \multicolumn{2}{|c|}{$\rho=-1$} & \multicolumn{2}{|c|}{$\rho=0$} & \multicolumn{2}{|c|}{$\rho=1$} \\\hline
   \backslashbox{Schedule}{Objective}& M-V  & P-C & P-C  &P-C & M-V  & P-C\\\hline
   M-V optimal& 0 & 98 &0  & 7.9 & 0 & 2.8\\\hline
   P-C optimal & 34 & 0& 5.2 &0  & 1.9 &0 \\\hline
 \end{tabular}
 \end{table}

\begin{figure}[h!]
\begin{subfigure}{0.48\textwidth}
\includegraphics[scale=0.4]{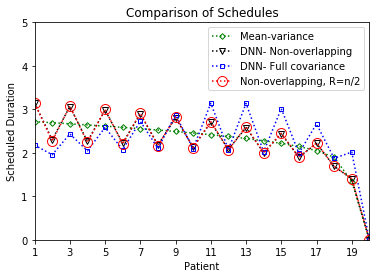}
\caption{ Correlation between patient 1 and 2 = correlation between
  patients 3 and 4 = $\ldots = \rho= 1$. For Full-covariance, all
  other correlations are set to zero. Maximum time available (T) is 45
  units. Optimal objective value = 25.0688, \textit{Mean-Variance}
  bound= 25.6151, DNN relaxation bound= 25.1534}
 \label{fig:schedules-mean-blk-rho-1}
 \end{subfigure}~~~
 \begin{subfigure}{0.48\textwidth}
\includegraphics[scale=0.4]{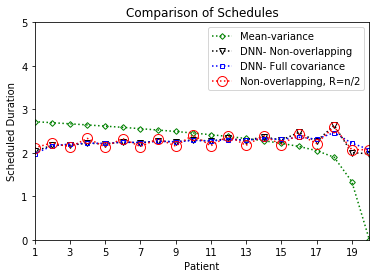}
\caption{Correlation between patient 1 and 2 = correlation between
  patients 3 and 4 = $\ldots = \rho= -1$. For Full-covariance, all
  other correlations are st to zero. Maximum time available (T) is 45
  units. Optimal objective value = 4.1162, \textit{Mean-Variance} bound= 25.6151, DNN
  relaxation bound= 4.2290 }
 \label{fig:schedules-mean-blk-rho--1}
 \end{subfigure}

 \begin{subfigure}{0.48\textwidth}
\includegraphics[scale=0.4]{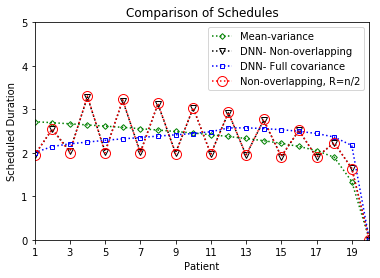}
\caption{Correlation between patient 1 and 2 = correlation between
  patients 3 and 4 = $\ldots = \rho= 0$. For Full-covariance, all
  other correlations are set to zero. Maximum time available (T) is 45 units.
  Optimal objective value=19.7474, \textit{Mean-Variance} bound= 25.6151 , DNN
  relaxation bound= 19.8607}
 \label{fig:schedules-mean-blk-rho-0}
 \end{subfigure}~~~
  \begin{subfigure}{0.48\textwidth}
\includegraphics[scale=0.4]{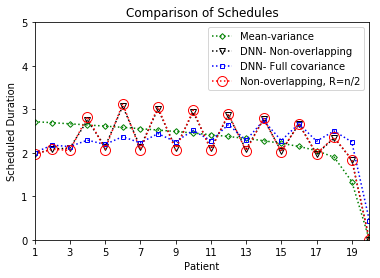}
\caption{ Correlations between patient 1 and 2 = correlations between
  patients 3 and 4 = $\ldots = \rho= -0.5$. For Full-covariance, all
  other correlations are set to zero. Maximum time available (T) is 45
  units.  Exact bound= 14.6842, \textit{Mean-Variance} bound=   25.6151, DNN
  relaxation bound=    14.7904}
 \label{fig:schedules-mean-blk-rho-0-large-t}
 \end{subfigure}
 \caption{Optimal schedules under knowledge of non-overlapping
   moments. 20 patients all with mean 2 and standard deviation 0.5.}
 \label{fig:schedules-expts}
\end{figure}



\subsection*{Acknowledgements} The research of the first and the second author was partially supported by the MOE Academic Research Fund Tier 2
grant T2MOE1706, ``On the Interplay of Choice, Robustness and Optimization in Transportation'' and the SUTD-MIT
International Design Center grant IDG21700101 on ``Design of the Last Mile Transportation System: What Does the
Customer Really Want?''. The authors would like to thank Teo Chung-Piaw (NUS) for providing some useful references on this research.
\bibliographystyle{abbrv}
\bibliography{partial_correlation_dro_ref}

\begin{thebibliography}{10}

\bibitem{AlmSorkin2002}
S.~E. Alm and G.~B. Sorkin.
\newblock Exact expectations and distributions for the random assignment
  problem.
\newblock {\em Combinatorics, Probability and Computing}, 11(3):217–248,
  2002.

\bibitem{Anstreicher2010}
K.~M. Anstreicher and S.~Burer.
\newblock Computable representations for convex hulls of low-dimensional
  quadratic forms.
\newblock {\em Mathematical Programming}, 124:33--43, 2010.

\bibitem{Ball1995}
M.~O. Ball, C.~J. Colbourn, and J.~S. Provan.
\newblock Network reliability.
\newblock {\em Handbooks in operations research and management science},
  7:673--762, 1995.

\bibitem{ben2009robust}
A.~Ben-Tal, L.~El~Ghaoui, and A.~Nemirovski.
\newblock {\em Robust Optimization}.
\newblock Princeton University Press, 2009.

\bibitem{Berge1967}
C.~Berge.
\newblock Some classes of perfect graphs.
\newblock {\em Graph Theory and Theoretical Physics}, F. Harary Ed:155--166,
  1967.

\bibitem{Bertsimas2010}
D.~Bertsimas, X.~V. Doan, K.~Natarajan, and C.-P. Teo.
\newblock Models for minimax stochastic linear optimization problems with risk
  aversion.
\newblock {\em Mathematics of Operations Research}, 35(3):580--602, 2010.

\bibitem{Bertsimas2004}
D.~Bertsimas, K.~Natarajan, and C.-P. Teo.
\newblock Probabilistic combinatorial optimization: Moments, semidefinite
  programming, and asymptotic bounds.
\newblock {\em SIAM Journal on Optimization}, 15(1):185--209, 2004.

\bibitem{Bertsimas2006}
D.~Bertsimas, K.~Natarajan, and C.-P. Teo.
\newblock Persistence in discrete optimization under data uncertainty.
\newblock {\em Mathematical Programming}, 108(2):251--274, Sep 2006.

\bibitem{Bertsimas2018}
D.~Bertsimas, M.~Sim, and M.~Zhang.
\newblock Adaptive distributionally robust optimization.
\newblock {\em o appear in Management Science}, 2018.

\bibitem{BomzeCDL17}
I.~M. Bomze, J.~Cheng, P.~J.~C. Dickinson, and A.~Lisser.
\newblock A fresh {CP} look at mixed-binary {QP}s: new formulations and
  relaxations.
\newblock {\em Math. Program.}, 166(1-2):159--184, 2017.

\bibitem{Bomze2002}
I.~M. Bomze and E.~De~Klerk.
\newblock Solving standard quadratic optimization problems via linear,
  semidefinite and copositive programming.
\newblock {\em J. of Global Optimization}, 24(2):163--185, Oct. 2002.

\bibitem{Burer2009}
S.~Burer.
\newblock On the copositive representation of binary and continuous nonconvex
  quadratic programs.
\newblock {\em Mathematical Programming}, 120(2):479--495, 2010.

\bibitem{Burer2015}
S.~Burer.
\newblock A gentle, geometric introduction to copositive optimization.
\newblock {\em Mathematical Programming}, 151(1):89--116, Jun 2015.

\bibitem{cambanis1976inequalities}
S.~Cambanis, G.~Simons, and W.~Stout.
\newblock Inequalities for {E} k (x, y) when the marginals are fixed.
\newblock {\em Zeitschrift f{\"u}r Wahrscheinlichkeitstheorie und verwandte
  Gebiete}, 36(4):285--294, 1976.

\bibitem{Delage2010}
E.~Delage and Y.~Ye.
\newblock Distributionally robust optimization under moment uncertainty with
  application to data-driven problems.
\newblock {\em Operations Research}, 58(3):595--612, 2010.

\bibitem{Dickinson2014}
P.~J.~C. Dickinson and L.~Gijben.
\newblock On the computational complexity of membership problems for the
  completely positive cone and its dual.
\newblock {\em Computational Optimization and Applications}, 57(2):403--415,
  Mar 2014.

\bibitem{Doan2012}
X.~V. Doan and K.~Natarajan.
\newblock On the complexity of nonoverlapping multivariate marginal bounds for
  probabilistic combinatorial optimization problems.
\newblock {\em Operations Research}, 60(1):138--149, 2012.

\bibitem{doi:10.1287/opre.10.6.808}
D.~R. Fulkerson.
\newblock Expected critical path lengths in pert networks.
\newblock {\em Operations Research}, 10(6):808--817, 1962.

\bibitem{galichon2016optimal}
A.~Galichon.
\newblock {\em Optimal Transport Methods in Economics}.
\newblock Princeton University Press, 2016.

\bibitem{GRONE1984109}
R.~Grone, C.~R. Johnson, E.~M. SÃ¡, and H.~Wolkowicz.
\newblock Positive definite completions of partial hermitian matrices.
\newblock {\em Linear Algebra and its Applications}, 58:109 -- 124, 1984.

\bibitem{Hagstrom1988}
J.~N. Hagstrom.
\newblock Computational complexity of pert problems.
\newblock {\em Networks}, 18(2):139--147, 1988.

\bibitem{Grani}
G.~A. Hanasusanto and D.~Kuhn.
\newblock Conic programming reformulations of two-stage distributionally robust
  linear programs over wasserstein balls.
\newblock {\em Operations Research}, 66(3):849--869, 2018.

\bibitem{Jiang2017}
R.~Jiang, S.~Shen, and Y.~Zhang.
\newblock Integer programming approaches for appointment scheduling with random
  no-shows and service durations.
\newblock {\em Operations Research}, 65(6):1638--1656, 2017.

\bibitem{Kong2013}
Q.~Kong, C.-Y. Lee, C.-P. Teo, and Z.~Zheng.
\newblock Scheduling arrivals to a stochastic service delivery system using
  copositive cones.
\newblock {\em Operations Research}, 61(3):711--726, 2013.

\bibitem{Kuhn1955}
H.~W. Kuhn.
\newblock The {H}ungarian method for the assignment problem.
\newblock {\em Naval Research Logistics Quarterly}, 2(1‐2):83--97, 1955.

\bibitem{Laurent2009}
M.~Laurent.
\newblock {\em Matrix completion problems Matrix Completion Problems}, pages
  1967--1975.
\newblock Springer US, Boston, MA, 2009.

\bibitem{Lorentz1953}
G.~G. Lorentz.
\newblock An inequality for rearrangements.
\newblock {\em The American Mathematical Monthly}, 60(3):176--179, 1953.

\bibitem{Mak2015}
H.-Y. Mak, Y.~Rong, and J.~Zhang.
\newblock Appointment scheduling with limited distributional information.
\newblock {\em Management Science}, 61(2):316--334, 2015.

\bibitem{Mohring2001}
R.~H. M{\"o}hring.
\newblock Scheduling under uncertainty: Bounding the makespan distribution.
\newblock In {\em Computational Discrete Mathematics}, pages 79--97. Springer,
  2001.

\bibitem{Munkres1957}
J.~Munkres.
\newblock Algorithms for the assignment and transportation problems.
\newblock {\em Journal of the Society for Industrial and Applied Mathematics},
  5(1):32--38, 1957.

\bibitem{natarajan2017reduced}
K.~Natarajan and C.-P. Teo.
\newblock On reduced semidefinite programs for second order moment bounds with
  applications.
\newblock {\em Mathematical Programming}, 161(1-2):487--518, 2017.

\bibitem{Natarajan2011}
K.~Natarajan, C.-P. Teo, and Z.~Zheng.
\newblock Mixed 0-1 linear programs under objective uncertainty: {A} completely
  positive representation.
\newblock {\em Operations Research}, 59(3):713--728, 2011.

\bibitem{Padberg1989}
M.~Padberg.
\newblock The boolean quadric polytope: Some characteristics, facets and
  relatives.
\newblock {\em Math. Program.}, 45(1):139--172, Aug. 1989.

\bibitem{Pardalos1993}
P.~M. Pardalos and K.~G. Ramakrishnan.
\newblock On the expected optimal value of random assignment problems:
  Experimental results and open questions.
\newblock {\em Computational Optimization and Applications}, 2(3):261--271,
  1993.

\bibitem{ParrilloThesis2000}
P.~A. Parrillo.
\newblock {\em Structured Semidefinite Programs and Semi-algebraic Geometry
  Methods in Robustness}.
\newblock PhD thesis, California Institute of Technology, 2000.

\bibitem{penrose_1955}
R.~Penrose.
\newblock A generalized inverse for matrices.
\newblock {\em Mathematical Proceedings of the Cambridge Philosophical
  Society}, 51(3):406–413, 1955.

\bibitem{Pitowsky1991}
I.~Pitowsky.
\newblock Correlation polytopes: Their geometry and complexity.
\newblock {\em Mathematical Programming}, 50(1):395--414, 1991.

\bibitem{rao1972}
C.~R. Rao and S.~K. Mitra.
\newblock Generalized inverse of a matrix and its applications.
\newblock In {\em Proceedings of the Sixth Berkeley Symposium on Mathematical
  Statistics and Probability, Volume 1: Theory of Statistics}, pages 601--620,
  Berkeley, Calif., 1972. University of California Press.

\bibitem{Rose1970}
D.~J. Rose.
\newblock Triangulated graphs and the elimination process.
\newblock {\em Journal of Mathematical Analysis and Applications}, 32(3):597 --
  609, 1970.

\bibitem{doi:10.1002/net.3230070407}
A.~W. Shogan.
\newblock Bounding distributions for a stochastic pert network.
\newblock {\em Networks}, 7(4):359--381, 1977.

\bibitem{SeDuMi}
J.~Sturm.
\newblock Using {SeDuMi} 1.02, a {MATLAB} toolbox for optimization over
  symmetric cones.
\newblock {\em Optimization Methods and Software}, 11--12:625--653, 1999.

\bibitem{Toh1999}
K.~Toh, M.~Todd, and R.~Tutuncu.
\newblock {SDPT3} --- a matlab software package for semidefinite programming.
\newblock {\em Optimization Methods and Software}, 11:545--581, 1999.

\bibitem{Tutuncu2003}
R.~Tutuncu, K.~Toh, and M.~Todd.
\newblock Solving semidefinite-quadratic-linear programs using sdpt3.
\newblock {\em Mathematical Programming Ser.}, B(95):189--217, 2003.

\bibitem{VanSlyke1963}
R.~M. Van~Slyke.
\newblock Monte carlo methods and the pert problem.
\newblock {\em Operations Research}, 11(5):839--860, 1963.

\bibitem{Wiese2014}
W.~Wiesemann, D.~Kuhn, and M.~Sim.
\newblock Distributionally robust convex optimization.
\newblock {\em Operations Research}, 62(6):1358--1376, 2014.

\bibitem{XuB18}
G.~Xu and S.~Burer.
\newblock A data-driven distributionally robust bound on the expected optimal
  value of uncertain mixed 0-1 linear programming.
\newblock {\em Comput. Manag. Science}, 15(1):111--134, 2018.

\bibitem{Burer2018}
A.~K. Yang~Boshi and S.~Burer.
\newblock Quadratic programs with hollows.
\newblock {\em Mathematical Programming}, 170(2):541--552, 2018.

\bibitem{Zangwill1966}
W.~I. Zangwill.
\newblock A deterministic multi-period production scheduling model with
  backlogging.
\newblock {\em Management Science}, 13(1):105--119, 1966.

\bibitem{Zangwill1969}
W.~I. Zangwill.
\newblock A backlogging model and a multi-echelon model of a dynamic economic
  lot size production system—a network approach.
\newblock {\em Management Science}, 15(9):506--527, 1969.

\bibitem{Zuluaga2005}
L.~F. Zuluaga and J.~F. Peña.
\newblock A conic programming approach to generalized tchebycheff inequalities.
\newblock {\em Mathematics of Operations Research}, 30(2):369--388, 2005.

\end{thebibliography}

\end{document}